\documentclass[12pt]{amsart}

\usepackage{hyperref,ulem}

\usepackage{amsmath,amstext,amssymb,amsopn,amsthm}
\usepackage{url,verbatim}
\usepackage{mathtools,MnSymbol}
\usepackage{enumerate,afterpage}
\usepackage{subfig}

\usepackage{color,graphicx,scalerel}

\usepackage[margin=30mm]{geometry}
\usepackage{eucal,mathrsfs,dsfont}%gives nicer set-names. 

\usepackage[scr=boondoxo]{mathalpha}

\let\emph\relax % there's no \RedeclareTextFontCommand
\allowdisplaybreaks

\newcommand*\rd{\mathop{}\!\mathrm{d}}

\newcommand{\scrq}{\mathscr{q}} 
\newcommand{\scru}{\mathscr{u}} 
\newcommand{\scrw}{\mathscr{w}} 
\newcommand{\scrh}{\mathscr{h}}

\newcommand{\Dn}{\mathbf{D}_n}
\newcommand{\DnE}{\mathbf{D}_n^E}
\newcommand{\DnEy}{\mathbf{D}_n^E(y)}

\newcommand{\wh}{\widehat}
\newcommand{\ya}{y_{\mathrm{A}}}
\renewcommand{\P}{\mathbb{P}}

\newcommand{\E}{\mathbb{E}}

\newcommand{\cB}{\mathcal{B}} 
\newcommand{\bx}{\mathbf{x}} 
\newcommand{\by}{\mathbf{y}} 
\newcommand{\bv}{\mathbf{v}}

\newcommand{\bz}{\mathbf{z}} 
 
\newcommand{\bX}{\mathbf{X}} 
\newcommand{\bY}{\mathbf{Y}} 
\newcommand{\bV}{\mathbf{V}} 
\newcommand{\bZ}{\mathbf{Z}} 
\newcommand{\bD}{\mathbf{D}} 
\newcommand{\wx}{\widetilde x}

\newcommand{\R}{\mathbb{R}}

\newcommand{\Q}{\mathbb{Q}}

\newcommand{\eps}{\varepsilon}

\newcommand{\bone}{\mathbf{1}}

\newcommand{\prt}{\partial}

\theoremstyle{plain}
\newtheorem{theorem}{Theorem}[section]
\newtheorem{lemma}[theorem]{Lemma}

\newtheorem{proposition}[theorem]{Proposition}

\theoremstyle{definition}

\newtheorem{remark}[theorem]{Remark}

\theoremstyle{remark}

\DeclareMathOperator{\sign}{sign}

\DeclareMathOperator{\Vol}{Vol}
\DeclareMathOperator{\var}{Var}
\def\crn#1#2{{\vcenter{\vbox{
        \hbox{\kern#2pt \vrule width.#2pt height#1pt
           }
          \hrule height.#2pt}}}}

\newcommand{\formula}[2][nolabel]
{\ifthenelse{\equal{#1}{nolabel}}
 {\begin{align*} #2 \end{align*}}
 {\ifthenelse{\equal{#1}{}}
  {\begin{align} #2 \end{align}}
  {\begin{align} \label{#1} #2 \end{align}}
 }
}

\numberwithin{equation}{section}

\begin{document}

%
%                            ---------- o ----------
%

\title[Archimedes' principle]{Archimedes' principle for ideal gas}

\author{Krzysztof Burdzy and Jacek Ma\l{}ecki}

\address{KB: Department of Mathematics, Box 354350, University of Washington, Seattle, WA 98195}
\email{burdzy@uw.edu}

\address{JM: Department of Mathematics \\ Wroc{\l}aw University of Science and Technology \\ ul. Wybrze{\.z}e Wyspia{\'n}\-skiego 27 \\ 50-370 Wroc{\l}aw, Poland}
\email{jacek.malecki@pwr.edu.pl}

\thanks{KB's research was supported in part by Simons Foundation Grant 506732. J. Ma\l{}ecki was supported by the Polish National Science Centre (NCN) grant no. 2018/29/B/ST1/02030.}



\begin{abstract}
We prove Archimedes' principle for a macroscopic ball in ideal gas consisting of point particles with non-zero mass. The main result is an  asymptotic theorem, as the number of point particles goes to infinity and their total mass remains constant. We also show that, asymptotically, the gas has an exponential density as a function of height. We find the asymptotic inverse temperature of the gas. We derive an accurate estimate of the volume of the phase space using the local central limit theorem.
\end{abstract}

\maketitle

\section{Introduction}

There seems to be no rigorous proof of Archimedes' principle in the mathematical literature. The most likely reason for this omission is that Archimedes' principle is  trivial given a few natural  assumptions. The principle can be easily derived using the divergence theorem, assuming that the formula for the pressure as a function of height is known. The ``barometric formula'' which says that the pressure in gas has an exponential density as a function of height can be easily derived from the ideal gas law.
While this derivation of Archimedes' principle is sufficient for the scientific applications, one could ask whether the principle can be derived from a more fundamental model of the matter, as in Hilbert's 6-th problem. This is what we will do in the present article.
Perhaps the most significant difference between our model and the derivation of Archimedes' principle alluded to above is that the floating object is allowed to move in our case. We are not aware of an existing proof of Archimedes' principle, rigorous or not, based on a model with a moving floating object.

We will consider a container with a bounded base, vertical side walls and no top. The container will hold point particles (ideal gas) and a floating object in the shape of a ball and no internal structure (the mass will be uniformly spread over the ball). The spherical shape of the floating object allows us to avoid the discussion of the energy going into rotation---the collisions of the ball with point particles and the walls of the container will not induce  rotation of the ball.

The point particles and the ball will move according to Newton's laws in a gravitational field with constant acceleration.
We will assume  conservation of energy and momentum but this assumption does not have a unique interpretation in the case of  collisions of point particles with the infinitely heavy walls of the container. We will consider two types of reflections of point particles from the walls of the container: (i) specular reflections where the angle of reflection is equal to the angle of incidence, and (ii) random reflections according to the Lambertian distribution also known as the Knudsen law. We will assume that the system is in equilibrium so that its density is given by the microcanonical ensemble formula. We will prove that this distribution is the unique invariant measure  in case (ii). Simple examples show that there are multiple invariant measures if we assume specular reflections.

In our asymptotic results we will assume that the following objects and quantities are fixed: the container, the mass and radius of the macroscopic ball, the total mass of the gas (all point particles), the total (potential and kinetic) energy of all moving objects (point particles and the ball), and the gravitational acceleration. The number of point particles will go to infinity so the mass of a single particle will go to zero.

On the way to the main result, Archimedes' principle, we will derive a few other results that may have independent interest.

We will present an accurate formula for the volume of the phase space. Our calculation is based on the local central limit theorem. This theorem was proved long time ago but the literature is hard to follow so we hope that our short review of that literature will help those readers who might need this result in their own research.

The microcanonical ensemble formula is well known, see, e.g., 
\cite[Sect.~1.2]{Ruelle}. 
We could not find a version of the formula needed here in the literature so
a derivation was supplied in a parallel project \cite{fermi}. We will not reproduce that proof here but we will present a brief review.
 
We will prove a version of the barometric formula, i.e., we will show that the density of the gas has, asymptotically, an exponential density as a function of height. We will show that the parameter of the exponential distribution can be identified with the inverse temperature.

\subsection{Literature review} 

A version of Archimedes' principle was proved in \cite{archim} but that model was completely different from the present one. The ``gas'' consisted of hard spheres with strictly positive radius. Their centers moved according to independent Brownian motions (except for the collisions). The number of ``gas molecules'' was constant and the asymptotic theorems were proved by sending the ``gravitational acceleration'' to infinity.

Our article is concerned with a model in which a macroscopic object interacts with microscopic molecules according to Newtonian mechanics. In this sense, our model is closest to the ``piston problem''  proposed in \cite{Lieb99}. A large number of papers were inspired by \cite{Lieb99} and devoted to the piston problem; see, for example, \cite{Sin99,LPS,CL2002,CLS2002,LSC2002,NeiSin,Gorel,IS} and references therein. 
Several different models were considered in those papers. 
In one of the models, a piston moves along a tube and is bombarded by microscopic molecules from  both sides.

\subsection{Limitations of the model}

Our model is, obviously, an oversimplified representation of reality and there is no hope that it could be modified to be very realistic. Still, from the mathematical point of view, some aspects of the model might be worth generalizing in future research.

(i) In our model, the macroscopic ball has no internal structure so it has negligibly small heat capacity. It might be possible to analyze a model in which the ball is replaced with a ``balloon,'' i.e., an infinitely thin sphere holding inside (mobile) point particles  with different masses than those of the outside particles.

(ii) Since we assume that the gas is ideal, the point particles do not interact and, therefore, their collisions with the macroscopic ball are the only way in which they  can exchange   energy between each other. This is the only way in which the energy may become approximately equidistributed in the stationary regime.
To generalize our results to gases consisting of hard spheres with positive radius one would need accurate estimates of the volume of the phase space. The virial expansion might be useful in this context, see \cite[Sect. 8i]{Mayer} or \cite[Sect. 4.3]{Ruelle}. If this approach works, it will require a whole new set of calculations.

(iii) In our model, the container has a flat bottom and vertical side walls. Many of our calculations depend on this assumption.

\subsection{Organization of the paper}
The rigorous presentation of the model and the statement of our main results are in Section \ref{y19.2}. We will review known results on the local central limit theorem in Section \ref{y19.3}. Section \ref{y19.4} contains a review of the results on the microcanonical ensemble formula developed in a parallel paper.  We will derive an accurate estimate of the phase space volume in Section \ref{y19.5}. The proof of Archimedes' principle will be given in Section \ref{y19.6}. The inverse temperature will be identified  in Section \ref{y19.7}. The uniqueness of the stationary probability distribution under Lambertian reflections will be proved in Section \ref{y19.8}.

\section{Model and main results}\label{y19.2}
We will consider $n$ point particles (ideal gas) and one macroscopic hard ball
in a $d$-dimensional container $D$ with  vertical walls, a bounded base, extending upward to infinity in the vertical direction.  

Suppose that $d\geq 2$ and 
let $D_b\subset \R^{d-1}$ denote the bottom of the container $D=D_b\times[0,\infty)$ in $\R^d$. 
The radius of the $(n+1)$-st (macroscopic) ball will be fixed and denoted $R >0$. Obviously, the point particles will have radii equal to $0$.
We will assume that $D_b$ has a smooth boundary and satisfies the inner ball condition with radius $R$, i.e., for every point $x\in \prt D_b$, there is a unique $(d-1)$-dimensional open ball with radius $R$ inside $D_b$ whose boundary is tangent to $\prt D_b$ at $x$ and, moreover, $x$ is the only point in the intersection of the ball boundary and $\prt D_b$. We will also assume that a closed ball of radius $R$ fits in the interior of $D$, so that each point particle has room to move from below to above the ball, and vice versa.

We will assume that  the mass of the $i$-th point particle is $m_i=m/n>0$ for $i=1,\dots, n$, for some $m>0$. When we let $n\to \infty$ in our theorems, the total mass of point particles will remain constant and equal to $m$. The mass of the macroscopic ball will be denoted $M = m_{n+1}>0$.
We will assume that
the mass is evenly spread over the volume of the macroscopic ball.
The walls of the container will be assumed to  have infinite mass---this assumption is a way of specifying the meaning of ``totally elastic collisions'' of point particles and the ball with the walls of the container.  

We will assume that the point particles and the ball are moving within a gravitational field with the constant acceleration $g>0$ pointing downwards. The point particles will not interact with each other. They will reflect from the walls of the container and they will undergo totally elastic collisions with the ball. The collisions of the  ball with the walls of the container will be totally elastic.

Time will  be suppressed in the notation, except for Section \ref{y19.8}. 
We will assume that the system is in equilibrium. Random objects  will represent the state of the system at time 0 (or any other fixed time). 

Let $(X_i,Y_i)\in\R^d$ denote the random position of the  $i$-th point particle or the center of the macroscopic ball (for $i=n+1$), where $X_i\in \R^{d-1}$ represents  the horizontal coordinates and $Y_i\geq 0$ is the vertical coordinate ($Y_{n+1} \geq R$). By $V_i\in\R^d$ we will denote the random velocity of the $i$-th point particle or the ball. 

Let 
\begin{align}\label{a1.4}
\bx_k&= (x_1,\ldots,x_{k}), \qquad \by_k= (y_1,\ldots,y_{k}),\qquad \bv_k=(v_1,\ldots,v_{k}),\\
 \bX_{k}&= (X_1,\ldots,X_{k}), \qquad  \bY_{k}= (Y_1,\ldots,Y_{k}), \qquad \bV_{k}= (V_1,\ldots,V_{k}), \label{j24.4}
\end{align}
where $x_i \in \R^{d-1}$, $y_i \in \R_+$, and $v_i \in \R^d$.
In the notation given above, upper case letters represent random variables and lower case letters represent their values.

 We will consider $(\bX_{n+1},\bY_{n+1},\bV_{n+1})$ to be a random vector distributed according to the microcanonical ensemble formula \eqref{j10.4} stated below, although we will give different distributions to these random vectors in some of the proofs.

We will assume that the total energy of our system,  $E$, is fixed. Hence, a.s.,
\begin{align}\label{y19.1}
E= \sum_{i=1}^{n+1} \left( m_igY_i+\frac{m_i||V_i||^2}{2}\right)\/,
\end{align}
with the convention that the zero level represents zero potential energy.
We will always assume that $E>MgR $ so that the ball and point particles cannot rest motionless at the bottom.

A ball in $\R^d$ with center $(x,y)$ and radius $r$ will be denoted $\cB((x,y), r)$.
 Let
\begin{align}\label{j10.1}
&D_b' = \{x\in D_b: \cB((x, y), R)\subset D
\text{  for all  } y>R\},\\
\label{j9.4}
 &\Dn = \Big\{( \bx_n,x_{n+1},\by_n,y_{n+1},\mathbf{v}_{n+1})\in D_b^{n}\times D_b'\times \R_+^n \times [R,\infty) \times \R^{nd}: \\
 & \qquad \qquad( x_k,  y_k) \in D\setminus \cB((x_{n+1},y_{n+1}),R), \ k=1,\dots n\Big\},\notag\\
&\DnE = \left\{(\mathbf{x}_{n+1},\by_{n+1},\mathbf{v}_{n+1})\in\Dn: \sum_{i=1}^{n+1}\left( m_iy_ig+\frac 1 2 m_iv_i^2\right)=E\right\} .\notag\\
\label{eq:DnEy:defn}
&\DnEy = \bigg\{(\mathbf{x}_{n},x_{n+1},\by_{n})\in D_b^{n}\times D_b'\times \R_+^n: \sum_{i=1}^{n+1} m_iy_ig\leq E, \\
&\qquad \qquad \qquad( x_k,  y_k) \in D\setminus \cB((x_{n+1},y),R), \ k=1,\dots n\bigg\}.\notag
\end{align}

Let  $\mu_{\by_{n+1}}(\rd x)$
denote 
the uniform probability measure on the sphere in $\R^{(n+1)d}$ (so that the sphere is
$((n+1)d-1)$-dimensional), 
centered at the origin, with the radius
$\left(2E-2\sum_{i=1}^{n+1} m_iy_ig\right)^{1/2}$.
Consider the following measure $\P_n$ on $\DnE$, 
\begin{align}\label{j10.4}
\P_n&(\rd \mathbf{x}_{n+1} \rd  \by_{n+1} \rd \mathbf{v}_{n+1})\\
& = C\left(E-\sum_{i=1}^{n+1} m_iy_ig\right)^{((n+1)d-2)/2}
\mu_{\by_{n+1}}\left(\frac{\rd v_1}{\sqrt{m_1}},\ldots,\frac{\rd v_{n+1}}{\sqrt{m_{n+1}}}\right)\rd \mathbf{x}_{n+1}\rd  \by_{n+1},\notag
\end{align}
where $C$ is the normalizing constant so that $\P_n$ is a probability measure. The measure $\P_n$ is a special case of the ``microcanonical ensemble formula,'' see, e.g.,  \cite[Sect.~1.2]{Ruelle}.

We will say that a point particle undergoes a Lambertian reflection from a surface or that the particle reflects according to the Knudsen law if the reflection occurs at a point where the inner normal to the surface is uniquely defined, the probability density of the angle between the reflected trajectory and the normal vector is proportional to the cosine of the angle, and the law is invariant under rotations about the normal vector. Note that the Lambertian distribution of the  outgoing velocity vector is independent from the incoming velocity vector, except that the two velocities have the same norm, so that  energy is conserved. 
This law for  random reflections was considered by Lambert in the context of light reflection from rough surfaces (\cite{L1760}) and by Knudsen (\cite{K1934}) as a model for gas molecule reflections.
By \cite[Cor. 3.2]{fermi} (which can be derived from \cite[Thm. 4.1]{Plakh} or \cite[Thm. 2.2]{ABS}), the Lambertian reflection law is the unique random reflection law that does not depend on the angle of incidence and is consistent with the standard (specular) reflection law from a rough surface consisting of small crystals with smooth reflecting surfaces.

\begin{theorem}\label{j18.1}
(i) The measure $\P_n$ (microcanonical ensemble formula) is invariant for the dynamical system defined above and two types of reflections: 

(a) totally elastic reflections between any pair of objects, 

(b) independent Lambertian reflections for point particles reflecting from the container walls (including the bottom) and totally elastic reflections between any other pair  of objects.

(ii) In case (b), $\P_n$ is the unique non-degenerate stationary probability distribution for the system of point particles and the macroscopic ball. The following are the only classes of ``degenerate'' invariant distributions:

(1) Invariant distributions such that no point particle ever hits a wall of the container.

(2) Invariant distributions such that at least one point particle has no energy, so that it is resting at the  bottom of the container. 
\end{theorem}

The proof of Theorem \ref{j18.1} will be given in Sections \ref{y19.4} and \ref{y19.8}.

\begin{remark}
(i) Part (ii) of Theorem \ref{j18.1} is not true under assumption (a); see Remark \ref{a1.1} (i).

(ii) We call invariant distributions in (1) degenerate because Lambertian reflections are never activated and the system has no opportunity to mix.
See Remark \ref{a1.1} (ii) for an example of such a distribution.

(iii)
Invariant distributions in (2) are degenerate because in this case the number of point particles is effectively less than $n$. 
\end{remark}

We will often use integrals of the form 
$\int_{D\setminus\cB((\wx, y), R)} \lambda e^{-\lambda r} \rd x \rd r $. In cases like this (and similar), $\wx$ should be interpreted as any point in $D_b'$, $\rd x$ will represent Lebesgue measure in $\R^{d-1}$, and $\rd r$ will represent Lebesgue measure in $\R$. Note that the integral does not depend on $\wx$ as long as $\wx \in D_b'$.

\begin{theorem}\label{j24.1}
Fix $d,D,R,M,m,g$ and $E$. 
Consider the equations
\begin{align}\label{a20.2}
 &M = m\frac
 {\int_{\cB((\wx, y), R)} \lambda e^{-\lambda r} \rd x \rd r}
 {\int_{D\setminus\cB((\wx, y), R)} \lambda e^{-\lambda r} \rd x \rd r } ,\\
&\frac{dmg}{2\lambda} + mg \frac
{ \int_{D\setminus\cB((\wx, y), R)} r\lambda e^{-\lambda r} \rd x \rd r} 
{ \int_{D\setminus\cB((\wx, y), R)} \lambda e^{-\lambda r} \rd x \rd r} 
+ Mg y = E, \label{j27.10}
\end{align}
with unknowns $y\geq R$ and $\lambda>0$.

(i)
There exists a unique $\lambda_*$ which satisfies \eqref{j27.10} with $y=R$. 

(ii) Suppose that 
\begin{align}\label{a8.12}
 &M < m\frac
 {\int_{\cB((\wx, R), R)} \lambda_* e^{-\lambda_* r} \rd x \rd r}
 {\int_{D\setminus\cB((\wx, R), R)} \lambda_* e^{-\lambda_* r} \rd x \rd r }.
\end{align}
Then \eqref{a20.2}-\eqref{j27.10} have  a unique solution that will be denoted $(\ya, \lambda_A)$. Moreover $R< y_A< E/(Mg)$.

\end{theorem}

The proof of Theorem \ref{j24.1} will be given in Section \ref{y19.5}.

For $\wx \in D_b'$, $y'\geq R$ and $\lambda >0$,
let $ \nu_{\wx,y',\lambda}$ be the probability distribution  defined by
\begin{align}\label{j24.3}
\nu_{\wx,y',\lambda}(\rd x,\rd y)=\frac {
\bone_{D\setminus \cB((\wx,y'),R)}(x,y) \lambda e^{-\lambda y}}
{\int_{D\setminus\cB((\wx, y'), R)} \lambda e^{-\lambda r} \rd z \rd r}\rd x\rd y
, \qquad x\in D_b, y>0 .
\end{align}

Let $\delta_{(X_i, Y_i)}$ denote the probability measure on $D$ which consists of a single atom at $(X_i,Y_i)$. The normalized (probability) empirical distribution $\Q_n$ of the gas (point particles) is defined as 
\begin{align}\label{j13.1}
\Q_n = \frac 1 n \sum_{i=1}^n \delta_{(X_i, Y_i)}  .
\end{align}

\begin{theorem}\label{a20.4}
(Archimedes' principle)
Fix $d,D,R,M,m,g$ and $E$.
Assume that the distribution of $(\bX_{n+1},\bY_{n+1}, \bV_{n+1})$ is  $\P_n$. Recall the notation from Theorem \ref{j24.1} and assume that \eqref{a8.12} holds.

(i) For every $\eps>0$,
\begin{align}\label{j8.1}
\lim_{n\to \infty} \P_n( |Y_{n+1} -\ya|>\eps)=0.
\end{align}

(ii) 
The marginal distribution of $X_{n+1}$ under $\P_n$ is uniform in $D_b'$. Given $\{X_{n+1}=x\}$, the conditional distribution of  $\Q_n$ converges to $\nu_{x,\ya,\lambda_A}$ weakly, in probability as $n\to \infty$.

\end{theorem}

The proof of Theorem \ref{a20.4} will be given in Section \ref{y19.7}.

\begin{remark}
(i)
Theorem \ref{a20.4} is a mathematical representation of Archimedes' principle.
Part (ii) of the theorem identifies the limiting 
empirical distribution of gas molecules when $n\to \infty$. 
The gas density is constant on the horizontal hyperplanes and it is exponential as a function of the height, outside the ball.

Part (i) of Theorem \ref{a20.4} and \eqref{a20.2} say that the ball is likely to float at the height such that the weight of the ball is equal to the weight of the displaced gas (i.e., \eqref{a20.2} holds), assuming that the gas is distributed as in part (ii).

(ii)
We need an extra equation \eqref{j27.10} to identify 
$\ya$ and $\lambda_A$  uniquely. 
The equation is an expression of  additivity  and conservation of energy. 
The three terms on the left hand side represent (asymptotically) the kinetic energy of the gas, the potential energy of the gas, and the potential energy of the ball. The kinetic energy of the ball is negligibly small asymptotically.

The form of
the first term on the left hand side of \eqref{j27.10}, i.e., 
$dmg/(2\lambda_A)$, representing the kinetic energy of the gas, is a manifestation of the virial theorem, which states how the energy is distributed between potential and kinetic forms.  The virial theorem is by now a classical result, discovered by Claussius in 1870 (see, for example \cite{Collins}).

Formula \eqref{j27.10} reduces to
\begin{align*} %\label{j28.5}
\frac{dmg}{2\lambda_A} + 
\frac{mg} {\lambda_A}  = E
\end{align*}
in the case when there is no ball, i.e., $R=0$ and $M=0$. In other words,  the ratio of potential to kinetic energy is $2/d$ for pure gas.

Alternatively one could say that
$dmg/(2\lambda_A)$  represents the kinetic energy, hence the heat energy, because $\lambda_A$ can be identified with the inverse temperature (see \eqref{j15.5} for a precise formula).

(iii) The heuristic meaning of the first claim of Theorem \ref{j24.1} is that there is a unique asymptotic distribution of gas if the ball rests at the bottom of the container. The second part says that if the weight of the ball is smaller than the weight of the gas displaced by the ball when it is placed at the bottom then, asymptotically, there is a unique level  at which the ball will float and the corresponding unique distribution of the gas.
\end{remark}

\section{Local Central Limit Theorem}\label{y19.3}

For any random variable $A$, let $f_A$ denote the density of $A$ (if it exists).
Suppose that random variables $\xi_k$, $k\geq 1$, are i.i.d. with $\E \xi_k =0$, $ E\xi_k^2 =1$ and density $f_{\xi_k}$.  Let  $M_3 = \E |\xi_k|^3 $ and
$S_n= \sum_{k=1}^n \xi_k/(\sqrt{n} \sigma)$. Let $\varphi(x)$ denote the standard normal density, i.e., $\varphi(x) = \frac 1 {\sqrt{2\pi}} e^{-x^2/2}$.

The following result is  Theorem 1 in \cite{SiraSaha}.

\begin{lemma}\label{j10.2}
There exists an absolute constant $C$ such that if
$f_{\xi_k}(x) \leq C_1$ for all $x\in \R$ then for all $n\geq 1$,
\begin{align}\label{j9.2}
\sup_{x\in\R}
\left| f_{S_n}(x) - \varphi(x)\right|
\leq \frac {C M_3^2 \max(1, C_1^3)}{\sqrt{n} }.
\end{align}
\end{lemma} 

\begin{remark}
There is a considerable literature on the local central limit theorem but the results are hard to extract from that literature for a number of reasons, including  absence of some journals from accessible libraries and lack of proofs in a number of publications.

The bound in
Lemma \ref{j10.2} is given in Remark 5, Section 4, Chapter VII in \cite{Petrov}. The bound is given there without a proof and it is attributed to \cite{Saha}. Unfortunately, \cite{Saha} does not contain a proof.
A similar bound is a special case of Theorem 1 in \cite{Serva} but, once again, that paper does not contain a proof.

One can derive Lemma \ref{j10.2} from Theorem 4 in \cite{Statu}.

The book \cite{Petrov} contains several versions of the local central limit theorem, see, e.g., Theorem 15 in Chapter VII. However, in each of these theorems the error is given in the form $o(1/\sqrt{n})$. One could extract a bound of the form $c/\sqrt{n}$ with an explicit formula for $c$ depending on the moments of the summands from the proof but that would be a very tedious task.

\end{remark}

\section{Microcanonical ensemble formula}\label{y19.4}

\begin{proof}[Proof of Theorem \ref{j18.1} (i)]
The formula \eqref{j10.4} is a version of a well known ``microcanonical ensemble formula,'' see, e.g.,  \cite[Sect.~1.2]{Ruelle}.
We could not find a rigorous proof that this distribution is invariant  in the  contemporary literature. A proof has been given in \cite{fermi}, a parallel project. The result proved in \cite[Prop. 5.3]{fermi} implies  that 
\begin{equation*}%\label{j10.5}
f(\bx_{n+1},\by_{n+1},\bv_{n+1}):=
\left(E-\sum_{i=1}^{n+1} m_iy_ig\right)^{((n+1)d-2)/2}
\frac{ \rd \mu_{\by_{n+1}}}{\rd  \bv_{n+1}}
\left(\frac{v_1}{\sqrt{m_1}},\ldots,\frac{v_{n+1}}{\sqrt{m_{n+1}}}\right)
\end{equation*}
is the density of an invariant measure on the whole space $\R^{2(n+1)d}$.

We claim that this density restricted to $\bD_n$ is invariant for our dynamical system with reflections. It suffices to show that the specular and Lambertian reflections between the point particles, the ball and the walls of the container leave the measure invariant. 
For the specular reflections, the proof is essentially identical to that of the proof of \cite[Prop. 2.3]{fermi}. The result can be extended to Lambertian reflections of point particles from the wall, as shown in \cite[Prop. 3.3, Remark 3.4]{fermi} (see the remarks preceding and following that result).
\end{proof}

\section{Phase space volume}\label{y19.5}

We will use two different sets of notations for some quantities.
While this may create some confusion, we will explain why this convention has some advantages.

For  $n=0,1,2,\ldots$, $\gamma>0$ and $y\geq R$ we let 
\begin{align}
\label{eq:hn:def}
   \scrh_n=\scrh_n(\gamma,y) = \int_{D\setminus \cB((\wx,y),R)}\gamma r^ne^{-\gamma r}\rd x\rd r\/,
\end{align}
where $\wx \in D_b'$. 
The definition of and notation for $\scrh_n$ lead to the following very simple formula, obtained using the dominated convergence theorem, 
\begin{align}
\label{eq:hn:der}
   \dfrac{\prt \scrh_n}{\prt \gamma} = \frac{\scrh_n}{\gamma}-\scrh_{n+1}\/.
\end{align}

Some of the functions $\scrh_n$ have physical meaning, so we find it easier to memorize them if different notation is used in some cases. Specifically,
\begin{align}\label{j27.1}
\scrq(\gamma,y) &= \int_{D\setminus \cB((\wx,y),R)} \gamma e^{-\gamma r} \rd x \rd r = \scrh_0(\gamma,y),\\
\scru(\gamma,y) &=  \frac 1 {\scrq(\gamma,y)}
 \int_{D\setminus\cB((\wx,y),R)} r\gamma e^{-\gamma r} \rd x \rd r
 = \frac{\scrh_1(\gamma,y)}{\scrh_0(\gamma,y)},\label{j28.2}\\
\scrw(\gamma,y) &= \int_{D\setminus \cB((\wx,y),R)} \gamma r^2 e^{-\gamma r} \rd x \rd r = \scrh_2(\gamma,y).\label{a1.5}
\end{align}

We will prove in Lemma  \ref{j24.2} that for given $u>0$ and $y\geq R$, the equation $u=\scru(\lambda,y)$ uniquely defines $\lambda=\lambda(u,y)$. We let
\begin{align}\label{a3.1}
h_n(u,y)&= \scrh_n(\lambda(u,y),y),\quad
q(u,y)= \scrq(\lambda(u,y),y),\quad
w(u,y)= \scrw(\lambda(u,y),y).
\end{align}

\begin{lemma}\label{y21.1}
For all  $y'\geq R$ and $\gamma>0$,
\begin{align}\label{y22.3}
&\frac 1 8  (1- e^{-1/2}) \leq\frac {\scrq(\gamma,y')} {|D_b|} \leq 1,\\
& \frac 1 {200} \leq  \gamma \scru(\gamma,y') \leq \frac 8 {1- e^{-1/2}}<21,\label{y22.4}\\
& 2^{-12} e^{-1} \leq \gamma^2 \left(\frac{\scrw(\gamma,y')}{\scrq(\gamma,y')} - \scru(\gamma,y')^2\right)  \leq \frac {16} {1- e^{-1/2}} ,\label{y22.5}\\
&
2^{-21}<\frac{1- e^{-1/2}} {16} \frac 1 {200^2} \leq 
\frac{\scru(\gamma,y')^2}{\scrw(\gamma,y')/\scrq(\gamma,y') - \scru(\gamma,y')^2} \leq \frac { 2^{18} e} {(1- e^{-1/2}) ^2}
< 2^{23}.\label{y22.6}
\end{align}
\end{lemma}

\begin{proof}
Let $\cB_*=\cB((\wx,y),R)$.
Elementary geometry shows that if
\begin{align*}
A= \left\{ x\in D_b: |x- \wx| < \frac{\sqrt{3}}2 R \right\}
\end{align*}
then $(D_b \setminus A ) \times (0,y'-R/2) \in D\setminus \cB_*$. 
If $\sigma(r)$ denotes the volume of a $(d-1)$-dimensional ball then, using our assumption that $d\geq 2$,
\begin{align}\label{y21.2}
\frac{|D_b \setminus A|}{|D_b|} = 1 - \frac{| A|}{|D_b|}
\geq 1 - \frac {\sigma\left(\sqrt{3}R/2\right)}{\sigma(R)} 
= 1 - \left(\frac {\sqrt{3}}{2}\right)^{d-1} \geq
1 - \frac {\sqrt{3}}{2} >\frac 1 8 .
\end{align}
We have 
\begin{align}\label{y21.3}
\scrq = \int_{D\setminus \cB_*} \gamma  e^{-\gamma r} \rd x \rd r
\geq \int_{D_b} \rd x \int_{y'+ R}^\infty \gamma  e^{-\gamma r} \rd r
= |D_b| \exp( - \gamma (y'+R)).
\end{align}
By \eqref{y21.2},
\begin{align}\label{y21.4}
\scrq &= \int_{D\setminus \cB_*} \gamma  e^{-\gamma r} \rd x \rd r
\geq \int_{D_b \setminus A} \rd x \int_0^{y'- R/2} \gamma  e^{-\gamma r} \rd r
= |D_b \setminus A|(1- \exp( - \gamma (y'-R/2)))\\
& \geq \frac 1 8 |D_b| (1- \exp( - \gamma (y'-R/2))).\notag
\end{align}
Recall that $y'\geq R$.
If $\gamma \leq 1/y'$ then \eqref{y21.3} yields
\begin{align}\label{y21.5}
\scrq \geq |D_b| \exp( - \gamma (y'+R))
\geq |D_b| \exp( - (1/y') (y'+R))
= |D_b| \exp( - 1 - R/y') \geq |D_b| e^{-2}.
\end{align}
In the case when $\gamma \geq 1/y'$, we use \eqref{y21.4} to see that
\begin{align*}
\scrq &\geq  \frac 1 8 |D_b| (1- \exp( - \gamma (y'-R/2)))
\geq  \frac 1 8 |D_b| (1- \exp( - (1/y') (y'-R/2)))\\
&= \frac 1 8 |D_b| (1- \exp( - 1 + R/(2y')))
\geq \frac 1 8 |D_b| (1- \exp( - 1 + 1/2))
= \frac 1 8 |D_b| (1- e^{-1/2}). 
\end{align*}
Since $e^{-2} > \frac 1 8  (1- e^{-1/2})$, the above estimate and \eqref{y21.5} give
\begin{align}\label{y22.7}
\scrq \geq \frac 1 8 |D_b| (1- e^{-1/2}).
\end{align}
This proves the lower bound in \eqref{y22.3}. The upper bound follows directly from the definition of $\scrq$.

We use \eqref{y22.7} to  derive the upper bound in \eqref{y22.4} as follows,
\begin{align}\label{y21.6}
\scru &=  \frac {
 \int_{D\setminus\cB_*} r\gamma e^{-\gamma r} \rd x \rd r}
{\int_{D\setminus \cB_*} \gamma e^{-\gamma r} \rd x \rd r}
= \frac 1 \scrq \int_{D\setminus\cB_*} r\gamma e^{-\gamma r} \rd x \rd r
\leq \frac 1 \scrq \int_{D} r\gamma e^{-\gamma r} \rd x \rd r\\
&= \frac 1 \scrq |D_b| \frac 1 \gamma
\leq \frac 8 {(1- e^{-1/2}) \gamma}.\notag
\end{align}

Let $W=(W_1,W_2)$ be a random vector with the distribution $\nu= \nu_{\wx,y',\gamma}$ defined in \eqref{j24.3}, where $W_1\in D_b$ and $W_2>0$. Note that
\begin{align}\label{y22.2}
\scrw/\scrq - \scru^2 =\var W_2 .
\end{align}
In view of \eqref{y21.2} and \eqref{y22.3}, 
the marginal density $f_{W_2}(y)$ satisfies
\begin{align}\label{y22.8}
f_{W_2}(y) &\leq \frac {|D_b|} \scrq \gamma e^{-\gamma y}
\leq \frac 8 {1- e^{-1/2}} \gamma e^{-\gamma y}, \qquad y >0,\\
\label{y21.7}
f_{W_2}(y) &\geq \gamma e^{-\gamma y}, \qquad y \geq y'+R,\\
f_{W_2}(y) &\geq \frac 1 8 \gamma e^{-\gamma y}, \qquad y \leq y'-R/2.
\label{y21.8}
\end{align}

It follows from \eqref{y22.8} that
\begin{align*}
\scrw/\scrq - \scru^2 =\var W_2  \leq \E W_2^2
= \int_0^\infty y^2 f_{W_2}(y) \rd y
\leq \int_0^\infty y^2 \frac 8 {1- e^{-1/2}} \gamma e^{-\gamma y} \rd y
= \frac 8 {1- e^{-1/2}} \frac 2 {\gamma^2}.
\end{align*}
This gives the upper bound in \eqref{y22.5}.

Suppose that $1/\gamma \leq y'$ and let $y'' = 1/\gamma$. 
Since $y'\geq R$ and $y''\leq y'$, we have $ y''/2 \leq y'-R/2$. Hence, we can apply \eqref{y21.8} to $y\in(0, y''/2)$ and obtain
\begin{align}\label{y22.9}
\scru &= \E W_2 \geq \int_{0}^{y''/2} y f_{W_2}(y) \rd y
\geq \int_{0}^{y''/2} y \frac 1 8 \gamma e^{-\gamma y} \rd y
= \frac 1{8\gamma} (1-e^{-y'' \gamma/2} (y'' \gamma/2+1))\\
&= \frac 1{8\gamma} (1-(3/2)e^{-1/2} )> \frac 1 {200 \gamma}.\notag
\end{align}

Suppose that $1/\gamma \geq y'$ and let $y'' = 1/\gamma$. 
Since $y'\geq R$ and $y''\geq y'$, we have $2 y'' \geq y'+R$. Hence, we can apply \eqref{y21.7} to $y>2 y''$ and obtain
\begin{align*}%\label{y22.9}
\scru &= \E W_2 \geq \int_{2y''}^\infty y f_{W_2}(y) \rd y
\geq \int_{2y''}^\infty y  \gamma e^{-\gamma y} \rd y
= \frac 1{\gamma} e^{-2y'' \gamma} (2y'' \gamma+1))
= \frac 1{\gamma} 3e^{-2} > \frac 1 {3 \gamma}.\notag
\end{align*}
This and \eqref{y22.9} yield the lower bound in \eqref{y22.4}.

Suppose that $1/\gamma \leq y'$ and let $y'' = 1/\gamma$. Then for any $a_0\in \R$ there exists an interval $(a_1,a_2)\subset (0, y''/2)$ such $a_2-a_1 \geq y''/8$ and for any $y\in(a_1,a_2)$, we have
$|y - a_0| \geq  y''/8$. We will apply this observation with $a_0 = \E W_2$. 
Since $y'\geq R$ and $y''\leq y'$, we have $a_2 \leq y''/2 \leq y'-R/2$. Hence, we can apply \eqref{y21.8} to $y\in(a_1,a_2)$ and obtain
\begin{align}\label{y22.1}
\var W_2 &\geq \int_{a_1}^{a_2} |y - \E W_2| ^2 f_{W_2}(y) \rd y
\geq \int_{a_1}^{a_2} (y''/8) ^2 \frac 1 8 \gamma e^{-\gamma y} \rd y
\geq (a_2-a_1) (y''/8) ^2 \frac 1 8 \gamma e^{-\gamma a_2}\\
&\geq (y''/8) ^3 \frac 1 8 (1/y'') e^{-\gamma y''}
= (y'')^2 2^{-12} e^{-1} = 2^{-12} e^{-1} \frac 1 {\gamma^2} .\notag
\end{align}

Next suppose that $1/\gamma \geq y'$ and let $y'' = 1/\gamma$.
Then for any $b_0\in \R$ there exists an interval $(b_1,b_2)\subset (2 y'', 6y'')$ such $b_2-b_1 \geq y''$ and for any  $y\in(b_1,b_2)$, we have
$|y - b_0| \geq  y''$.
We will apply this observation with $b_0 = \E W_2$. 
Since $y'\geq R$ and $y''\geq y'$, we have $b_1 \geq 2y'' \geq y'+R$. Hence, we can apply \eqref{y21.7} to $y\in(b_1,b_2)$ and obtain
\begin{align*}
\var W_2 &\geq \int_{b_1}^{b_2} |y - \E W_2| ^2 f_{W_2}(y) \rd y
\geq \int_{b_1}^{b_2} (y'') ^2  \gamma e^{-\gamma y} \rd y
\geq (b_2-b_1) (y'') ^2  \gamma e^{-\gamma b_2}\\
&\geq (y'') ^3  (1/y'') e^{-6\gamma y''}
= (y'')^2  e^{-6} =  e^{-6} \frac 1 {\gamma^2} .
\end{align*}
Since $e^{-6} > 2^{-12} e^{-1}$, the above estimate, \eqref{y22.2} and \eqref{y22.1} imply that
\begin{align*}
\scrw/\scrq - \scru^2 =\var W_2 &\geq  2^{-12} e^{-1} \frac 1 {\gamma^2} .
\end{align*}
This yields the lower bound in \eqref{y22.5}.

The bound in \eqref{y22.6} follows from \eqref{y22.4} and \eqref{y22.5}.
\end{proof}

\begin{lemma}\label{j24.2}

(i) For any $y\geq R$, the function $\lambda \to \scru(\lambda,y)$ 
is strictly decreasing.

(ii) For every $u>0$ and $y\geq R$
there exists a unique
$\lambda>0$ satisfying $ u = \scru(\lambda,y)$.
\end{lemma}

\begin{proof}
Recall $\scrh_n$ defined in \eqref{eq:hn:def} and formula \eqref{eq:hn:der}. We have
\begin{align*}
  \dfrac{\prt }{\prt \lambda}\left(\frac {
 \int_{D\setminus\cB_*} r\lambda e^{-\lambda r} \rd x \rd r}
{\int_{D\setminus \cB_*} \lambda e^{-\lambda r} \rd x \rd r}\right) = \dfrac{\prt }{\prt \lambda}\left(\frac{\scrh_1}{\scrh_0}\right)
= \frac{\scrh_1^2-\scrh_2\scrh_0}{\scrh_0^2}\/,
\end{align*}
which is strictly negative by the Cauchy-Schwarz inequality and, consequently, $\lambda \to \scru(\lambda,y) =\scrh_1(\lambda,y)/\scrh_0(\lambda,y)$ is strictly decreasing for every $y\geq R$. This proves (i).

 Note that 
\begin{align*}
\scrh_1(\lambda,y) =
\int_{D\setminus \cB_*}\lambda r e^{-\lambda r}\rd x\rd r
\leq \int_{D}\lambda r e^{-\lambda r}\rd x\rd r = \frac {|D_b|} \lambda.
\end{align*}
This and \eqref{y22.3} imply that
\begin{align*}
 \scru(\lambda,y)=  \frac{\scrh_1(\lambda,y)}{\scrh_0(\lambda,y)} &\leq \frac {|D_b|} {\lambda \scrq} \leq   
   \frac{8}{\lambda(1-e^{1/2})}\to 0\/, \quad \text{  as  }\lambda\to\infty.
\end{align*} 
On the other hand, using \eqref{y22.3} again, when $\lambda \to 0$,
\begin{align*} 
 \scru(\lambda,y)=	 \frac{\scrh_1(\lambda,y)}{\scrh_0(\lambda,y)} & = \frac {\scrh_1(\lambda,y)}{\scrq(\lambda,y)}\geq \frac{|D_b|} \scrq \int_{y+R}^\infty \lambda re^{-\lambda r}\rd r 
	 \geq \frac{1}{\lambda} \int_{\lambda(y+R)}se^{-s}\rd s \to \infty\/. 
\end{align*} 
We conclude that for given $u>0$ and $y\geq R$ there exists a unique $\lambda=\lambda(u,y)>0$ such that $ u = \scru(\lambda,y)$ holds.
\end{proof}
	
\begin{lemma}
  \label{lem:der}
  Recall that $\lambda(u,y)$ denotes the solution to $ u = \scru(\lambda,y)$, and recall \eqref{j27.1}-\eqref{a3.1}. We have
	\begin{align}
	  \label{eq:l:der:u}
		  \dfrac{\prt \lambda}{\prt u} &= \frac{q}{qu^2-w}\/,\\
		\label{eq:l:der:y}
		  \dfrac{\prt \lambda}{\prt y} &= \dfrac{\lambda u |D_b|-q}{qu^2-w}\/,\\
		\label{eq:p:der:u}
		  \dfrac{\prt q}{\prt u} &= \frac{q}{\lambda}(1-\lambda u)\dfrac{\prt \lambda}{\prt u}\/,\\
		\label{eq:p:der:y}
		 \dfrac{\prt q}{\prt y} &= \frac{q}{\lambda}(1-\lambda u)\dfrac{\prt \lambda}{\prt y}+\lambda (|D_b|-q)\/,\\
		\label{eq:pel:der:u}
		\dfrac{\prt }{\prt u}\left(\frac{qe^{\lambda u}}{\lambda}\right) &= qe^{\lambda u}\/.
	\end{align}
\end{lemma}
\begin{proof}
Comparing \eqref{eq:hn:def} and \eqref{j27.1}-\eqref{a3.1}, we see that $q=h_0$, $uh_0=h_1$ and $w=h_2$.

 Using   \eqref{eq:hn:der} and differentiating the identity
\begin{align}
   \label{eq:h01:rel}
    u\scrh_0(\lambda(u,y),y) = \scrh_1(\lambda(u,y),y)
\end{align}
with respect to $u$ we obtain
\begin{align*}
  &\scrh_0+u\left(\frac{\scrh_0}{\lambda}-\scrh_1\right)\dfrac{\prt \lambda}{\prt u} = \left(\frac{\scrh_1}{\lambda}-\scrh_2\right)\dfrac{\prt \lambda}{\prt u},\\
	&\dfrac{\prt \lambda}{\prt u} = \scrh_0\left(\frac{\scrh_1}{\lambda}-\scrh_2-\frac{u\scrh_0}{\lambda}+u\scrh_1\right)^{-1} = \frac{q}{qu^2-w}\/, 
\end{align*}
where we used the facts that $w=h_2$ and $uq=uh_0=h_1$. 
This proves \eqref{eq:l:der:u}.

We prove \eqref{eq:p:der:u} as follows,
\begin{align*}
   \dfrac{\prt q}{\prt u} = \dfrac{\prt \scrh_0}{\prt \lambda}\dfrac{\prt \lambda}{\prt u} = \left(\frac{\scrh_0}{\lambda}-\scrh_1\right)\dfrac{\prt \lambda}{\prt u}  = \frac{q}{\lambda}(1-\lambda u)\dfrac{\prt \lambda}{\prt u}\/.
\end{align*}

We use the above formula in the following proof of
\eqref{eq:pel:der:u},
\begin{equation*}
  \dfrac{\prt}{\prt u}\left(\frac{qe^{\lambda u}}{\lambda}\right) =\frac{e^{\lambda u}}{\lambda^2}\left(q(1-\lambda u)\dfrac{\prt \lambda}{\prt u}+q\lambda \left(\lambda+u\dfrac{\prt \lambda}{\prt u}\right)-q\dfrac{\prt \lambda}{\prt u}\right)=qe^{\lambda u}\/.
\end{equation*}

Since $\int_0^\infty \lambda r^n e^{-\lambda r} \rd r = n!/\lambda^n$ and the volume of the $(d-1)$-dimensional sphere with radius $\rho$ is $\pi^{(d-1)/2} \rho^{d-1}/ \Gamma(\frac{d+1}{2})$,
\begin{align*}
  \scrh_n(\lambda,y) &= \frac{|D_b|n!}{\lambda^n} - \frac{\pi^{(d-1)/2}}{\Gamma(\frac{d+1}{2})}\int_{y-R}^{y+R}\lambda r^ne^{-\lambda r}(R^2-(y-r)^2)^{(d-1)/2}\rd r\\
	&= \frac{|D_b|n!}{\lambda^n} - \frac{\pi^{(d-1)/2}}{\Gamma(\frac{d+1}{2})}e^{-\lambda y}\int_{-R}^{R}\lambda (s+y)^ne^{-\lambda s}(R^2-s^2)^{(d-1)/2}\rd s.
\end{align*}
Consequently, by the dominated convergence theorem, we obtain
\begin{align}
 \dfrac{\prt \scrh_n}{\prt y} &= \lambda\left(\frac{|D_b|n!}{\lambda^n}-\scrh_n\right)-n\left(\frac{|D_b|(n-1)!}{\lambda^{n-1}}-\scrh_{n-1}\right) .
\end{align}
In particular,
\begin{align}\label{a1.6}
  \dfrac{\prt \scrh_0}{\prt y} = \lambda(|D_b|-\scrh_0)\/,\quad \dfrac{\prt \scrh_1}{\prt y} = (|D_b|-\lambda \scrh_1)-(|D_b|-\scrh_0) = \scrh_0 -\lambda \scrh_1\/.
\end{align}
Differentiation of \eqref{eq:h01:rel} with respect to $y$ gives
\begin{align*}
  &u\left(\dfrac{\prt \scrh_0}{\prt \lambda}\dfrac{\prt \lambda}{\prt y}+\dfrac{\prt \scrh_0}{\prt y}\right) = \dfrac{\prt \scrh_1}{\prt \lambda}\dfrac{\prt \lambda}{\prt y}+\dfrac{\prt \scrh_1}{\prt y},\\
	&\dfrac{\prt \lambda}{\prt y} = \left(\dfrac{\prt \scrh_1}{\prt y}-u\dfrac{\prt \scrh_0}{\prt y}\right)\left(u\dfrac{\prt \scrh_0}{\prt \lambda}-\dfrac{\prt \scrh_1}{\prt \lambda}\right)^{-1} \/.
\end{align*}
We now use \eqref{eq:hn:der} and \eqref{a1.6} to see that
\begin{align*}
	&\dfrac{\prt \lambda}{\prt y} 
	= \frac{\lambda u |D_b|-q}{qu^2-w}\/.
\end{align*}
This proves \eqref{eq:l:der:y}.

We apply \eqref{eq:hn:der} and \eqref{a1.6} once again to prove \eqref{eq:p:der:y},
\begin{align*}
   \dfrac{\prt q}{\prt y} = \dfrac{\prt \scrh_0}{\prt \lambda}\dfrac{\prt \lambda}{\prt y}+\dfrac{\prt \scrh_0}{\prt y} = \frac{q}{\lambda}(1-\lambda u)\dfrac{\prt \lambda}{\prt y}+\lambda(|D_b|-q)\/.
\end{align*} 
\end{proof}

\begin{proof}[Proof of Theorem \ref{j24.1}]
In view of \eqref{j27.1}-\eqref{j28.2},
the equations \eqref{a20.2}-\eqref{j27.10} can be written in this form,
\begin{align}\label{j28.3}
 &M = K(\lambda,y) := m\frac {|D_b| - \scrq( \lambda,y)}
 {\scrq( \lambda,y) }
 = m\left(\frac{|D_b| }{ \scrq( \lambda,y)} -1\right) ,\\
&G(\lambda,y):=\frac{dmg}{2\lambda} + mg \scru(\lambda,y) 
+ Mg y = E. \label{j28.4}
\end{align}

It is easy to see that all partial derivatives of any order of the functions 
$K(\lambda,y)$ and $G(\lambda,y)$
exist and are continuous. 

Consider any $y\geq R$ such that $Mgy < E$. For sufficiently small $\lambda >0$, $dmg/(2\lambda) >E$, so for some $\lambda>0$, $G(\lambda,R) >E$. 
Our assumption that $Mgy < E$ and  \eqref{y22.4} imply that  for very large $\lambda$, $G(\lambda,R ) < E$. By continuity of $G(\lambda,y)$, there exists $\lambda$ such that $G(\lambda,y) = E$. Let $\lambda_y$ denote the smallest $\lambda$ with this property.

Part (i) of the lemma holds true because we have assumed that $MgR < E$. Therefore, we can take $\lambda_*=\lambda_R$. 

We have $\lim_{y \uparrow E/(Mg)} \lambda _y = \infty$ because the term $Mgy$ in the formula for $G(\lambda,y)$ approaches $E$, so the first term, $dmg/(2 \lambda_y)$, must go to 0.

By the assumptions of part (ii) of the lemma, $K( \lambda_R,R)=K( \lambda_*,R) >M$. It is easy to see that
the function $q(\lambda,y)$ converges to $|D_b|$ when $\lambda \to \infty$, no matter how $y$ and $\lambda$ are related. Hence $\lim_{y\uparrow E/(Mg)} K(\lambda_y,y) =0$. By continuity of $K(\lambda,y)$, there exists $y$ such that $K(\lambda_{y},y)=M$. 
Let $y_A$ be the smallest $y$ with this property and let $\lambda_A = \lambda_{y_A}$. Note that $R < y_A < E/(Mg)$.

It remains to prove uniqueness of the solution $(y,\lambda)$ to \eqref{a20.2}-\eqref{j27.10}.

Let $z(y) = \frac{E-Mgy}{mg}$.
By Lemma \ref{j24.2}, the function 
$$
  \kappa_y(u) := u+\dfrac{d}{2\lambda(u,y)}-z(y)
$$	
is a strictly increasing function of $u$. By \eqref{y22.4}, $\kappa_y(0^+)=-z(y)$ and $\kappa_y(z(y))> 0$, so for every given $y\geq R$ there exists a unique $u=u(y)$ such that
\begin{align}\label{a3.2}
u(y)=z(y)-\dfrac{d}{2\lambda(u(y),y)}.
\end{align}
Comparing this formula to \eqref{j28.4}, we see that we must have $\scru(\lambda,y) = u(y)$. It will suffice to show that there is at most one $y$ such that $\scrq(\lambda(u(y),y),y)=q(u(y),y)$ satisfies \eqref{j28.3}.

Assuming that $u(y)$ satisfies \eqref{a3.2},
$$
  \dfrac{\rd  u}{\rd  y} = -\frac{M}{m}+\frac{d}{2\lambda^2}\left(\dfrac{\prt \lambda}{\prt u}\dfrac{\prt u}{\prt y}+\dfrac{\prt \lambda}{\prt y}\right).
$$
Hence,
\begin{align}\label{a7.10}
&   \dfrac{\rd  u}{\rd  y}\left(1-\frac{d}{2\lambda^2}\dfrac{\prt \lambda}{\prt u}\right) =  -\frac{M}{m}+\frac{d}{2\lambda^2}\dfrac{\prt \lambda}{\prt y},\\
&  \dfrac{\rd  u}{\rd  y} = \dfrac{-\dfrac{M}{m}+\dfrac{d}{2\lambda^2}\dfrac{\prt \lambda}{\prt y}}{1-\dfrac{d}{2\lambda^2}\dfrac{\prt \lambda}{\prt u}}.\notag
\end{align}
We use this formula and \eqref{eq:p:der:u}-\eqref{eq:p:der:y}
to see that
\begin{align}\label{j31.2}
   \dfrac{\rd }{\rd y}q(u(y),y) &= \dfrac{\prt q}{\prt u}\dfrac{\prt u}{\prt y}+\dfrac{\prt q}{\prt y}= \frac{q}{\lambda}(1-\lambda u) \left(\dfrac{\prt \lambda}{\prt u}\dfrac{\rd  u}{\rd  y}+\dfrac{\prt \lambda}{\prt y}\right)+\lambda(|D_b|-q)\\
&= \frac{q}{\lambda}(1-\lambda u) \left(\dfrac{\prt \lambda}{\prt u}\dfrac{-\dfrac{M}{m}+\dfrac{d}{2\lambda^2}\dfrac{\prt \lambda}{\prt y}}{1-\dfrac{d}{2\lambda^2}\dfrac{\prt \lambda}{\prt u}}+\dfrac{\prt \lambda}{\prt y}\right)+\lambda(|D_b|-q).\notag
\end{align}	

Lemma \ref{j24.2} (i) implies that $\prt \lambda/ \prt u <0$ so $1-\frac{d}{2\lambda^2}\frac{ \prt \lambda}{ \prt u}>0$. Since we are interested only in the sign of $\rd q/\rd y$, it will suffice to analyze $A := \frac{\rd q}{\rd y}\left(1-\frac{d}{2\lambda^2}\frac{ \prt \lambda}{ \prt u}\right)$.
Multiplying both sides of \eqref{j31.2} by  $1-\frac{d}{2\lambda^2}\frac{ \prt \lambda}{ \prt u}$, we obtain
\begin{align*}
   &A=\frac{q}{\lambda}(1-\lambda u) \left(\dfrac{ \prt \lambda}{ \prt u}\left(-\frac{M}{m}+\dfrac{d}{2\lambda^2}\dfrac{ \prt \lambda}{ \prt y}\right)+\dfrac{ \prt \lambda}{ \prt y}\left(1-\dfrac{d}{2\lambda^2}\dfrac{ \prt \lambda}{ \prt u}\right)\right)+\lambda(|D_b|-q)\left(1-\dfrac{d}{2\lambda^2}\dfrac{ \prt \lambda}{ \prt u}\right)\\
	&= \frac{q}{\lambda}(1-\lambda u)\left(-\frac{M}{m}\dfrac{ \prt \lambda}{ \prt u}+\dfrac{ \prt \lambda}{ \prt y}\right)+\lambda(|D_b|-q)\left(1-\dfrac{d}{2\lambda^2}\dfrac{ \prt \lambda}{ \prt u}\right)\\
	&= \frac{q}{\lambda(w-qu^2)}\left((1-\lambda u)\left(\frac{Mq}{m}+q-\lambda u |D_b|\right)+\frac{d}{2}(|D_b|-q)\right)+\lambda(|D_b|-q)\\
	&=\frac{q|D_b|}{\lambda(w-qu^2)}\left((1-\lambda u)^2+(1-\lambda u)\left(\frac{q}{|D_b|}(1+M/m)-1\right)+\frac{d}{2}\left(1-\frac{q}{|D_b|}\right)\right)+\lambda(|D_b|-q).
\end{align*}
If $(q/|D_b|)(1+M/m)-1=0$ then
\begin{align*}
   A=\frac{q|D_b|}{\lambda(w-qu^2)}\left((1-\lambda u)^2+\frac{d}{2}\left(1-\frac{q}{|D_b|}\right)\right)+\lambda(|D_b|-q).
\end{align*}
According to \eqref{y22.2}, $w-qu^2 >0$. Since we assume that $R>0$, $|D_b|-q>0$. It follows that $A>0$ and, therefore, $\rd q/\rd y>0$ if $(q/|D_b|)(1+M/m)-1=0$. In other words, $\rd q/\rd y>0$ whenever $q(u(y),y) = |D_b|/(1+M/m)$. A smooth function cannot cross a level multiple times if its derivative is strictly positive at every crossing point. We have proved that there is at most one $y$ such that $q(u(y),y)$ satisfies \eqref{j28.3}.
\end{proof}

Recall \eqref{a1.4} and for $u>0$, $y\geq R$ and $x\in D_b'$ let
\begin{align}\notag
 \bD_n(y,u) &=\Bigg \{( \bx_n,x_{n+1},\by_n)\in D_b^{n}\times D_b' \times \R^n:  \frac 1 n \sum_{i=1}^n y_i = u,\\
&\qquad ( x_k,  y_k) \in D\setminus \cB((x_{n+1},y),R), \ k=1,\dots n\Bigg\}, \notag\\
 \bD_n(x,y,u) &=\Bigg \{( \bx_n,\by_n)\in D_b^{n} \times \R^n:  \frac 1 n \sum_{i=1}^n y_i = u,\label{j13.2}\\
&\qquad ( x_k,  y_k) \in D\setminus \cB((x,y),R), \ k=1,\dots n\Bigg\} .\notag
\end{align}

\begin{proposition}\label{j11.2}
Consider $u>0$, $y\geq R$ and $x\in D_b'$. 
Let $\Vol_n(x,y,u)$ be the $(nd-1)$-dimensional volume of 
$\bD_n(x,y,u)$ and $\Vol_n(y,u)$ be the $((n+1)d-2)$-dimensional volume of 
$\bD_n(y,u)$. Let
$\lambda=\lambda(u,y)>0$ be the solution to $ u = \scru(\lambda,y)$
and recall $q=q(u,y)$ defined in \eqref{j27.1} and \eqref{a3.1}.
Then for some absolute constants $0<c_1,c_2,c_3,c_4<\infty$ and $n_1$, for all $n\geq n_1$,
\begin{align}\label{j19.1}
c_1|D_b| n^{1-n} 
\left( \frac{ q e^{\lambda u}}{\lambda}\right)^{n-1}\leq
\Vol_n(x,y,u) &\leq c_2|D_b| n^{1-n} 
\left( \frac{ q e^{\lambda u}}{\lambda}\right)^{n-1}, \\
 c_3 |D_b'|\cdot|D_b| n^{1-n} 
\left( \frac{ q e^{\lambda u}}{\lambda}\right)^{n-1}\leq
\Vol_n(y,u) &\leq  c_4  |D_b'|\cdot|D_b| n^{1-n} 
\left( \frac{ q e^{\lambda u}}{\lambda}\right)^{n-1}.\label{j10.6}
\end{align}
\end{proposition}

\begin{proof}
First note that \eqref{j10.6} is an immediate corollary of \eqref{j19.1} so it will suffice to prove \eqref{j19.1}.

Let $\cB_*=\cB((x,y),R)$,
\begin{align*}
H &=\{(z_1,t_1, \dots,z_n, t_n):( t_1+ \dots + t_n)/n=u; t_k>0,
z_k \in D_b,\text{  for } 1\leq k \leq n\},\\
H_*&=H \cap (D\setminus \cB_*)^n .
\end{align*}

Let $  (Z^{1 }_1,T^{1 }_1),\ldots,(Z^{1 }_{n},T^{1 }_n)$ be i.i.d., with $Z^{1 }_i \in \R^{d-1}$ and $T^{1 }_i >0$ for $1\leq i \leq n$. Assume that $T^{1 }_i$ and $Z^{1 }_i$ are independent, $Z^{1 }_i$ has the uniform distribution in $D_b$, and $T^{1 }_i$  has the exponential distribution with parameter $\lambda$.

Let $f_1(z_1,t_1, \dots,z_n, t_n)$ be the density of $\left((Z^{1 }_1,T^{1 }_1),\ldots,(Z^{1 }_{n},T^{1 }_n)\right)$.
We have 
\begin{align}\label{j23.3}
f_1(z_1,t_1, \dots,z_n, t_n)
= |D_b|^{-n} \lambda ^n \exp(-\lambda(t_1+ \dots + t_n)),
\end{align}
for $z_i \in D_b$, $t_i >0$, $1\leq i \leq n$. Note that the density $f_1$ is  constant on  $H$.

Let $(Z^{2 }_1,T^{2 }_1),\ldots,(Z^{2 }_{n},T^{2 }_n)$ be i.i.d., with $( Z^{2 }_i,  T^{2 }_i)$ being distributed as $( Z^{1 }_i,  T^{1 }_i)$ conditioned by 
$\{(Z^{1 }_i,T^{1 }_i)\notin \cB_*\}$.
  By the definition of $q$ and \eqref{j23.3}, the density $f_2$ of $((Z^{2 }_1,T^{2 }_1),\ldots,(Z^{2 }_{n},T^{2 }_n))$ is given by
\begin{align}\label{j23.4}
f_2(z_1,t_1, \dots,z_n, t_n)
= \left(\frac{q}{|D_b|}\right)^{-n} |D_b|^{-n} \lambda^n \exp(-\lambda(t_1+ \dots + t_n))
\end{align}
on the set $(D\setminus \cB_*)^n$.

Let $S^1_n = \frac 1 n \sum_{k=1}^n T^1_k = \sum_{k=1}^n (T^1_k/n)$. The distribution of $T^1_k/n$ is exponential with mean $(n\lambda)^{-1}$, so the distribution of $S^1_n$ is gamma with the density $f_{S^1_n}(s) =(( n\lambda)^n/\Gamma(n)) s^{n-1} e^{-n\lambda s}$. Hence,
\begin{align}\label{j23.5}
f_{S^1_n}(u) = \frac{( n\lambda)^n}{\Gamma(n)} u^{n-1} e^{-n\lambda u}.
\end{align}
Let $S^2_n = \frac 1 n \sum_{k=1}^n T^2_k $ and $T^3_j = T^2_j -u$ for $1\leq j \leq n$. 
It follows from the fact that $ u = \scru(\lambda,y)$ that $\E T^2_j=u$ and $\E T^3_j=0$. Let $\sigma^2 = \var T^2_j=\var T^3_j$ and $S^3_n = \frac 1 {n^{1/2}\sigma} \sum_{k=1}^n T^3_k =n^{1/2}\sigma^{-1}(S_n^2-u) $. 
We have
\begin{align*}
f_{S_n^2}(s)  = \sigma^{-1} n^{1/2} f_{S_n^3}(\sigma^{-1} n^{1/2}(s-u)).
\end{align*}
By Lemma \ref{j10.2}, 
\begin{align}\label{j23.6}
f_{S_n^2}(u) =\sigma^{-1} n^{1/2} \left( \frac 1 {\sqrt{2 \pi}} +A\right),
\end{align}
where
\begin{align}\label{y22.10}
|A| &\leq \frac {C \left(\E |T^3_j/\sigma|^3\right)^2 \max(1, C_1^3)}{\sqrt{n} },\\
C_1 &= \sup_{u\in \R} f_{T^3_j/\sigma}(u).\label{y22.11}
\end{align}
In view of \eqref{j23.3}, \eqref{j23.4}, \eqref{j23.5} and \eqref{j23.6}, 
\begin{align}\label{j23.8}
\frac{\Vol(H_*)}{\Vol(H)}
= \frac{q^n f_{S_n^2}(u)}{|D_b|^n f_{S_n^1}(u)}
= \frac{q^n\sigma^{-1} n^{1/2} \left( \frac 1 {\sqrt{2 \pi}} +A\right)}{|D_b|^n \frac{( n\lambda)^n}{\Gamma(n)} u^{n-1} e^{-n\lambda u}}.
\end{align}
The volume of  a regular $n$-simplex with unit side length is
$\sqrt{n}/((n-1)! 2^{(n-1)/2})$. The volume of  a regular $n$-simplex with the side length $\sqrt{2}$ is
$(\sqrt{2})^{n-1}\sqrt{n}/((n-1)! 2^{(n-1)/2}) =\sqrt{n}/(n-1)! $. 
Hence $\Vol(H) = |D_b|^n u^{n-1} \sqrt{n}/(n-1)!$.
This and \eqref{j23.8} imply that
\begin{align}\label{y22.12}
\Vol(H_*)
&=\Vol(H) \frac{q^n\sigma^{-1} n^{1/2} \left( \frac 1 {\sqrt{2 \pi}} +A\right)}{|D_b|^n\frac{( n\lambda)^n}{\Gamma(n)} u^{n-1} e^{-n\lambda u}}
=\frac{|D_b|^n u^{n-1} \sqrt{n}}{(n-1)! }\frac{q^n\sigma^{-1} n^{1/2} \left( \frac 1 {\sqrt{2 \pi}} +A\right)}{|D_b|^n\frac{( n\lambda)^n}{\Gamma(n)} u^{n-1} e^{-n\lambda u}}\\
&= 
\frac{ q^n\sigma^{-1} n \left( \frac 1 {\sqrt{2 \pi}} +A\right)}{( n\lambda)^n  e^{-n\lambda u}}
=\left( \frac{ q e^{\lambda u}}{\lambda}\right)^n
\left[
n^{1-n}\sigma^{-1}  \left( \frac 1 {\sqrt{2 \pi}} +A\right)\right].\notag
\end{align}

In this proof,
for any functions $a$ and $b$ of any number of variables, we will write $a\approx b$ to indicate that there exist universal constants $0< c', c'' < \infty$ such that $c' \leq a/b \leq c''$ for all values of the arguments.

By Lemma \ref{y21.1}, $q \approx |D_b|$, $\lambda u \approx 1$ and $\lambda \sigma = \lambda \sqrt{w/q - u^2}  \approx 1$, so \eqref{y22.12} implies that
\begin{align}\label{a12.1}
\Vol(H_*)\approx \left( \frac{ q e^{\lambda u}}{\lambda}\right)^{n-1}
\left[|D_b|
n^{1-n} \left( \frac 1 {\sqrt{2 \pi}} +A\right)\right].
\end{align}
Next we will  find an upper bound for $A$ defined in \eqref{j23.6} using \eqref{y22.10}-\eqref{y22.11}.

We use the bound $f_{T_j^2}(y) \leq \frac{|D_b|}{q}\lambda e^{-\lambda y}$ and the substitution $\lambda r = s$ to obtain
\begin{equation*}
\E |T^3_j/\sigma|^3
\leq \sigma^{-3}\frac{|D_b|}{q}\int_0^\infty |r-u|^3\lambda e^{-\lambda r}\rd r = \frac{1}{(\lambda \sigma)^3} \frac{|D_b|}{q} \int_0^\infty |s-u\lambda|^ 3 e^{-s}\rd s.
\end{equation*}
We have already pointed out that $\lambda \sigma  \approx 1$, $|D_b|/q\approx 1$ and $u \lambda \approx 1$. The last formula shows that $\E |T^3_j/\sigma|^3 \leq c_5< \infty$, where $c_5$ is a universal constant.

Recall that $\lambda \sigma  \approx 1$ and $|D_b|/q\approx 1$ to see that for some universal constant $c_6 < \infty$,
\begin{equation*}
   f_{T_j^3/\sigma}(s) \leq q^{-1}\sigma\lambda e^{-\lambda(s\sigma + u)}
   \leq \sigma\lambda q^{-1}\leq \frac{c_6}{|D_b|}.
\end{equation*}
This bound and $\E |T^3_j/\sigma|^3 \leq c_5$ imply that $|A| \leq c_7/\left(\sqrt{n}\min(1,|D_b|^3)\right)$ where $c_7$ is an absolute constant.

Consider the case $|D_b|= 1$. 
Then $|A| \leq c_7/\sqrt{n}$.
Let $n_1$ be so large that for $n\geq n_1$,
\begin{align*}
\frac 1 {2\sqrt{2 \pi}} <
\frac 1 {\sqrt{2 \pi}} + \frac {c_7}{\sqrt{n}}
<\frac 2 {\sqrt{2 \pi}}.
\end{align*}
This and \eqref{a12.1} prove \eqref{j19.1} in the case $|D_b| = 1$.

For  $|D_b|\ne 1$, we use scaling. If $x,y, D, R$ and $u$ are multiplied by $c_*>0$, it is elementary to verify that $\Vol_n(x,y,u)$, $q$, $\lambda$ and $u$ are rescaled by powers of $c_*$ such that \eqref{j19.1} remains true.
\end{proof}

\begin{remark}
We believe that Proposition \ref{j11.2} holds for all $n$. One could prove the bounds for $n<n_1$ using elementary estimates similar to those in the proof of Lemma \ref{y21.1}. We do not provide a proof because it is not needed for our main theorem.
\end{remark}

\section{Archimedes' principle}\label{y19.6}

In this section,
for any functions $a_1(\,\cdot \,)$ and $a_2(\,\cdot \,)$ of any arguments we will write
\begin{align}\label{a5.1}
a_1(\,\cdot \,) \approx a_2(\,\cdot \,)\qquad
\Longleftrightarrow \qquad
\exists 0< c_1, c_2<\infty: c_1 a_1(\,\cdot \,) \leq a_2(\,\cdot \,) \leq c_2 a_1(\,\cdot \,).
\end{align}
The constants $c_1$ and $c_2$ may depend only on ``fixed'' parameters in our model: $d,D,D_b,R,M,m,g$ and $E$.

\begin{proposition}\label{j13.3}
Recall the notation from  Theorem \ref{j24.1} and assume that \eqref{a8.12} holds. Recall \eqref{j28.2} and let $u_A = \scru(\lambda_A,y_A)$. For every $\eps >0$,
\begin{align*}
\lim_{n\to\infty} \P_n \left(|Y_{n+1} - y_A|> \eps
\text{  or  } \left| \sum_{i=1}^n Y_i - u_A\right| >\eps\right) =0.
\end{align*}
\end{proposition}

\begin{proof}

Recall notation from \eqref{eq:DnEy:defn}.
The following formula for the marginal density $ f_{Y_{n+1}}(y)$ of $Y_{n+1}$ follows from \eqref{j10.4},
\begin{equation}\label{a8.8}
  f_{Y_{n+1}}(y) = \frac{1}{Z_n}\int_{\DnEy}\left(E-mg\frac{1}{n}\sum_{i=1}^n y_i-Mgy\right)^{((n+1)d-2)/2}\rd \bx_{n+1}\rd \by_{n}\/,
\end{equation}
where $Z_n$ is the normalizing constant.
Let $z(y)=(E-Mgy)/(mg)$ and 
note that $z(y)  >0$ for $y\in[R, E/(Mg))$ because we have assumed that $E> MgR$.
Recall notation from  Proposition \ref{j11.2}. The proposition implies that 
\begin{align}
  f_{Y_{n+1}}(y) &= \frac{1}{Z_n}\int_0^{z(y)}\left(E-Mgy-mgu\right)^{((n+1)d-2)/2} \Vol_n(y,u)\rd u \label{a8.9}\\
	&\approx \frac{1}{Z_n'}|D_b| n^{1-n}
	\int_0^{z(y)}(z(y)-u)^{((n+1)d-2)/2}
\left(\frac{q(u,y)e^{\lambda(u,y) u}}{\lambda(u,y)}\right)^{n-1}\rd u\label{eq:fY:int}
\end{align}
where 
\begin{equation*}
Z_n' = 
|D_b| n^{1-n}
\int_R^{E/(Mg)}\rd y\int_0^{z(y)}(z(y)-u)^{((n+1)d-2)/2}
  \left(\frac{q(u,y)e^{\lambda(u,y) u}}{\lambda(u,y)}\right)^{n-1}\rd u\/.
\end{equation*} 
The integral in \eqref{eq:fY:int} can be written as
\begin{equation}
  \label{eq:integral}
 \int_0^{z(y)}\alpha_y^{n-1}(u) (z(y)-u)^{d-1}\rd u\/,
\end{equation}
where $\alpha_y(u) = (z(y)-u)^{d/2}q e^{\lambda u}/\lambda$. 

Let 
\begin{align}\label{a6.4}
\beta_y(u) = \lambda(u,y)-\frac d{2 (z(y)-u)}.
\end{align}
By \eqref{eq:pel:der:u},
\begin{equation}
  \label{eq:hy:der}
 \dfrac{\prt}{\prt u}\alpha_y(u) = (z(y)-u)^{d/2}\frac{qe^{\lambda u}}{\lambda}\left(\lambda-\frac{d}{2(z(y)-u)}\right) = \alpha_y(y)\beta_y(u)\/.
\end{equation}

By Lemma \ref{j24.2}, the function $u\to \lambda(u,y)$ is strictly decreasing. Hence  $\beta_y(u) $ is strictly decreasing. 
When $u\downarrow 0$, $\lambda \to \infty$ by \eqref{y22.4}. 
Thus $\beta_y(0^+)=\infty$ for $y\in[R, E/(Mg))$.
The  estimate \eqref{y22.4} implies that  $\beta_y(z(y)^-)=-\infty$.
We conclude that there is a unique $u_0=u_0(y)\in (0,z(y))$ such that  $\beta_y(u_0)=0$, $\beta_y(y)$ is positive on $(0,u_0)$ and negative on $(u_0,z(y))$. In view of \eqref{eq:hy:der}, $\alpha_y(u)$ attains its only maximum at $u_0$ on the interval $(0,u_0)$, and the maximum is strict. Note that
\begin{equation}\label{a6.5}
   \beta_y(u) = \beta_y(u)-\beta_y(u_0) = \lambda(u,y)-\lambda(u_0,y)+\frac{d}{2}\frac{u_0-u}{(z(y)-u_0)(z(y)-u)}\/.
\end{equation}
This, and the facts that the function $u\to \lambda(u,y)$ is  decreasing 
and $z(y) \geq z(y)-u >0$ for $u\in(0,z(y))$, imply that 
\begin{align}\label{a6.2} 
   |\beta_y(u)| &= -\beta_y(u) \geq \frac{d}{2z^2(y)}(u-u_0) =  \frac{d}{2z^2(y)}|u-u_0|\/,\quad u\in( u_0,z(y))\/,\\
	 |\beta_y(u)| &= \beta_y(u) \geq  \frac{d}{2z^2(y)}(u_0-u) =  \frac{d}{2z^2(y)}|u-u_0|\/,\quad u\in(0,u_0)\/.\label{a6.3}
\end{align}
To simplify notation, we will write $k=n-1$.
From \eqref{a6.2}-\eqref{a6.3}, we obtain for $u \notin [u_0-z(y)k^{-1/2},u_0+z(y)k^{-1/2}]$,
\begin{align*} %\label{a6.1}
 \frac{2 \sqrt{k} z(y) |\beta_y(u)|}{ d} =
\frac 2 d
 \frac{ z^2(y) |\beta_y(u)|}{ z(y)k^{-1/2}} \geq 
\frac 2 d \frac{ z^2(y) |\beta_y(u)|}{ |u-u_0|} & \geq 1\/.
\end{align*}
This, \eqref{eq:integral} and \eqref{eq:hy:der}, and the bound $z(y)-u<z(y)$ imply that
\begin{align}\label{a8.10}
&\left(\int_0^{u_0-z(y)k^{-1/2}}+\int_{u_0+z(y)k^{-1/2}}^{z(y)}\right)\alpha_y^{k}(u) (z(y)-u)^{d-1}\rd u\\
&\leq
\left(\int_0^{u_0-z(y)k^{-1/2}}+\int_{u_0+z(y)k^{-1/2}}^{z(y)}\right)\alpha_y^{k}(u) (z(y)-u)^{d-1}
\frac{2 \sqrt{k} z(y) |\beta_y(u)|}{ d} \rd u\notag\\
&=
  \frac{2z^{d}(y)\sqrt{k}}{d} \left(\int_0^{u_0-z(y)k^{-1/2}}+\int_{u_0+z(y)k^{-1/2}}^{z(y)}\right)\alpha_y^{k-1}(u)\alpha_y(u)\beta_y(u)\rd u\notag\\
&  =
   \frac{2z^{d}(y)}{d\sqrt{k}}\left(\alpha^{k}_y\left(u_0-\frac{z(y)}{\sqrt{k}}\right)+\alpha^{k}_y\left(u_0+\frac{z(y)}{\sqrt{k}}\right)-\alpha_y^{k}(0)-\alpha_y^{k}(z(y))\right) \leq \frac{4}{d\sqrt{k}}z^{d}(y)\alpha_y^{k}(u_0)\/.\notag
\end{align}
We combine this estimate with the observation that $z(y) \leq E/mg$ and the following bound
\begin{equation}\label{a8.6}
   \int_{u_0-z(y)k^{-1/2}}^{u_0+z(y)k^{-1/2}}\alpha_y^{k}(u)(z(y)-u)^{d-1}\rd u\leq \frac{2z^{d}(y)}{\sqrt{k}}\alpha_y^{k}(u_0)\/,
\end{equation}
to arrive at
\begin{equation}
  \label{eq:integral:upper}
	\int_0^{z(y)}\alpha_y^{k}(u)(z(y)-u)^{d-1}\rd u \leq  \left(2+\frac{4}{d}\right)\,\frac{z^{d}(y)}{\sqrt{k}}\alpha_y^{k}(u_0)\/.
\end{equation}

Since $\beta_y(u)\leq0$ for $u\in[ u_0,z(y))$, \eqref{a6.4} implies that, for $u\in[ u_0,z(y))$,
\begin{align*}
u\geq\frac{z(y)}{1+\frac d {2u\lambda(u,y)}}.
\end{align*}
This and \eqref{y22.4} yield, 
\begin{align}\label{a7.6}
   u&\geq  \frac{z(y)}{1+100 d},\qquad u\in[ u_0,z(y)).
\end{align}
Since $\beta_y(u_0)=0$, we obtain from  \eqref{a6.4},
\begin{align*}
u_0(y)=\frac{z(y)}{1+\frac d {2u_0(y)\lambda(u_0(y),y)}},
\end{align*}
and, by \eqref{y22.4}, 
\begin{align}
   u_0(y) &\leq \frac{z(y)} {1 + d/42} \leq \frac{21}{22} z(y). \label{a7.7}
\end{align}

Using \eqref{eq:l:der:u}, \eqref{y22.4}, \eqref{y22.5} and \eqref{a7.6}, we get for $u\in[ u_0,z(y))$,
\begin{align}\label{a7.1}
 \left|\dfrac{\prt \lambda}{\prt u}\right| &= -\dfrac{\prt \lambda}{\prt u} = \frac{1}{(w/p-u^2)\lambda^2}(\lambda u)^2\frac{1}{u^2} \leq 2^{12}e\cdot 21^2 \frac{(1+100d)^2}{z^2(y)}=\frac{c_1}{z^2(y)}\/,
\end{align}
where $c_1$ depends only on $d$.

For $u\in(u_0,u_0+z(y)k^{-1/2})$ and $k\geq 44^2=1936$, in view of \eqref{a7.7},
\begin{align*}
&\frac{z^2(y)}{(z(y)-u_0)(z(y)-u)} 
\leq \frac{z^2(y)}{\left(z(y)-\frac{21}{22} z(y)\right)\left(z(y)-\frac{21}{22}z(y) - z(y) k^{-1/2}\right)} \\
&\leq \frac 1 {(1/22)(1/44)} \leq 968.
\end{align*}
This estimate, \eqref{a6.5} and \eqref{a7.1} show that
for  $u\in(u_0,u_0+z(y)k^{-1/2})$ and $k\geq 1936$,
\begin{align}\notag
   -\beta_y(u) &= \lambda(u_0,y)-\lambda(u,y)+\frac{d}{2}\frac{u-u_0}{(z(y)-u_0)(z(y)-u)}\\
	&\leq \frac{u-u_0}{z^2(y)}\left(c_1+484d\right)\leq \frac{c_2}{z(y)\sqrt{k}},\label{a7.2}
\end{align}
where $c_2>0$ depends only on $d$. 
For  $u\in(u_0,u_0+z(y)k^{-1/2})$ and $k\geq 1936$,
\begin{align*}
z(y)-u \geq z(y)-\frac{21}{22}z(y) - z(y) k^{-1/2}\geq z(y)/44.
\end{align*}
This implies that, for  $u\in(u_0,u_0+z(y)k^{-1/2})$ and $k\geq 1936$,
\begin{align}\label{a7.3}
&\int_0^{z(y)}\alpha_y^{k}(u) (z(y)-u)^{d-1}\rd u\\
&\quad\geq
  \alpha_y^{k}(u_0)
\int_{u_0}^{u_0+z(y)k^{-1/2}}(z(y)-u)^{d-1} \left(1-\frac{\alpha_y(u_0)-\alpha_u(u)}{\alpha_y(u_0)}\right)^{k}\rd u\notag\\
&\quad\geq
  z^{d-1}(y) \alpha_y^{k}(u_0)\left(1/44\right)^{d-1}\int_{u_0}^{u_0+z(y)k^{-1/2}} \left(1-\frac{\alpha_y(u_0)-\alpha_u(u)}{\alpha_y(u_0)}\right)^{k}\rd u\/.\notag
\end{align}
Recall that $u_0$ is the maximum of $\alpha_y$, and also \eqref{eq:hy:der} and \eqref{a7.2}.
 There exists $\tilde{u}\in (u_0,u_0+z(y)k^{-1/2})$ such that, 
\begin{align*}
   \frac{\alpha_y(u_0)-\alpha_u(u)}{\alpha_y(u_0)} = -\frac{\alpha_y(\tilde{u})}{\alpha_y(u_0)}\beta_y(\tilde{u})(u-u_0)\leq \frac{c_2}{z(y)\sqrt{k}}(u-u_0)\leq \frac{c_2}{k}\/.
\end{align*}
Hence, for  $u\in(u_0,u_0+z(y)k^{-1/2})$ and $k\geq 1936$,
\begin{align*}
 \left(1-\frac{\alpha_y(u_0)-\alpha_u(u)}{\alpha_y(u_0)}\right)^{k}
 \geq c_3.
\end{align*}
We combine this with \eqref{a7.3} to see that
\begin{equation}\label{a7.8}
  \int_0^{z(y)}\alpha_{y}^k(y)(z(y)-u)^{d-1}\rd u\geq \frac{c_4}{\sqrt{k}}z^d(y)\alpha_y^k(u_0)\/.
\end{equation}

Let $\lambda_0=\lambda(u_0(y),y)$, $q_0=q(u_0(y),y)$ and
\begin{align}\label{a7.9}
   \psi(y) &= (2/d)^{d/2}\alpha_y(u_0)=(2/d)^{d/2}(z(y)-u_0(y))^{d/2} \frac{q_0e^{\lambda_0u_0}}{\lambda_0} \/.
\end{align}
Since $u_0(y)$ makes the right hand side of \eqref{a6.4} equal to 0,
\begin{align}\label{a7.11}
   \psi(y) &= (2/d)^{d/2} (z(y)-u_0(y))^{d/2} \frac{q_0e^{\lambda_0u_0}}{\lambda_0} = \frac{q_0e^{\lambda_0u_0}}{\lambda_0^{d/2+1}}\/.
\end{align}

Recall that $k=n-1$.
It follows from \eqref{eq:fY:int}, \eqref{eq:integral}, \eqref{eq:integral:upper}, \eqref{a7.8} and \eqref{a7.9}  that
\begin{equation}\label{a8.3}
   f_{Y_{n+1}}(y) \approx \frac{1}{Z_n''}\,\psi^{n-1}(y)z^{d}(y)\/,
\end{equation}
with the normalizing constant $Z_n''= \int_R^{E/(Mg)}\psi^{n-1}(y)z^{d}(y)\rd y$. 

Comparing \eqref{a3.2} and \eqref{a6.4}, we see that we can apply  \eqref{a7.10} to $u_0(y)$, i.e., 
\begin{equation}\label{a8.5}
   \dfrac{\rd  u_0}{\rd  y}\left(1-\frac{d}{2\lambda^2_0}\dfrac{\prt \lambda}{\prt u}\right) = -\frac{M}{m}+\frac{d}{2\lambda^2_0}\dfrac{\prt \lambda}{\prt y}\/.
\end{equation}
We use this formula, \eqref{eq:p:der:u}, \eqref{eq:p:der:y} and \eqref{a7.11} in the following calculation,
\begin{align*}
 \lambda_0^{d/2+2}&e^{-\lambda_0 u_0}\dfrac{\rd \psi}{\rd y}  =  \lambda_0\left(\dfrac{\prt q}{\prt u}\dfrac{\rd u_0}{\rd y}+\dfrac{\prt q}{\prt y}\right) +\lambda_0^2 q_0 \dfrac{\rd u_0}{\rd y}+q_0\left(\lambda_0u_0-\frac{d+2}{2}\right)\left(\dfrac{\prt \lambda}{\prt u}\dfrac{\rd u_0}{\rd y}+\dfrac{\prt \lambda}{\prt y}\right)\\
	& = \dfrac{\rd u_0}{\rd y}\left(\lambda_0\dfrac{\prt q}{\prt u}+q_0\left(\lambda_0u_0-\frac{d+2}{2}\right)\dfrac{\prt \lambda}{\prt u}+\lambda_0^2q_0\right)+\lambda_0\dfrac{\prt q}{\prt y}+ q_0\left(\lambda_0u_0-\frac{d+2}{2}\right)\dfrac{\prt \lambda}{\prt y}\\
	& = \dfrac{\rd u_0}{\rd y}\left(\lambda_0 \frac{q_0}{\lambda_0}(1-\lambda_0 u_0)\dfrac{\prt \lambda}{\prt u}
+q_0\left(\lambda_0u_0-\frac{d+2}{2}\right)\dfrac{\prt \lambda}{\prt u}+\lambda_0^2q_0\right)\\
&\qquad+\lambda_0\left( \frac{q_0}{\lambda_0}(1-\lambda_0 u_0)\dfrac{\prt \lambda}{\prt y}+\lambda_0 (|D_b|-q_0) \right)
+ q_0\left(\lambda_0u_0-\frac{d+2}{2}\right)\dfrac{\prt \lambda}{\prt y}\\
	& = \lambda_0^2q_0\dfrac{\rd u_0}{\rd y}\left(1-\frac{d}{2\lambda_0^2}\dfrac{\prt \lambda}{\prt u}\right) +\lambda_0^2(|D_b|-q_0)-\frac{dq_0}{2}\dfrac{\prt \lambda}{\prt y}\\
	& =\lambda_0^2q_0\left(-\frac{M}{m}+\frac{d}{2\lambda^2_0}\dfrac{\prt \lambda}{\prt y}\right)+\lambda_0^2(|D_b|-q_0)-\frac{dq_0}{2}\dfrac{\prt \lambda}{\prt y}\\
	& = \frac{\lambda_0^2 q_0}{m}\left(m\frac{|D_b|-q_0}{q_0}-M\right)\/.
\end{align*}
Recall the usual $\sign$ function that takes values $-1,0$ or $1$.
The formula given above implies that 
\begin{align}\label{a8.2}
\sign\left(\dfrac{\rd \psi(y)}{\rd y}\right)
= \sign\left(m\frac{|D_b|-q_0(y)}{q_0(y)}-M\right)
= \sign\left(m\frac
 {\int_{\cB((\wx, y), R)} \lambda_0 e^{-\lambda_0 r} \rd x \rd r}
 {\int_{D\setminus\cB((\wx, y), R)} \lambda_0 e^{-\lambda_0 r} \rd x \rd r } -M\right).
\end{align}
According to \eqref{a6.4} and the definition of $u_0(y)$,
\begin{align}\label{a8.1}
 \lambda(u_0(y),y)-\frac d{2 (z(y)-u_0(y))} =0.
\end{align}
This is equivalent to \eqref{j27.10}
(see also \eqref{j28.4} and \eqref{a3.2}).
In view of \eqref{a20.2}-\eqref{a8.12}, \eqref{a8.2} and \eqref{a8.1},  by Theorem \ref{j24.1} (ii), there exists a unique $y_A\in( R, E/(Mg))$ such that $\left.\dfrac{\rd \psi(y)}{\rd y}\right|_{y=y_A} =0$ and \eqref{a8.1} holds.
Hence
$\psi(y)$ attains its maximum at $y_A$ and it is strictly increasing on $(R,y_A)$ and strictly decreasing on $(y_A,E/(Mg))$. 

For later reference we note that the above argument also shows that $\lambda_0=\lambda(u_0(y_A),y_A) = \lambda_A$, with $\lambda_A$ as defined in Theorem \ref{j24.1}, and, therefore, $u_0(y_A) = u_A$, with $u_A$ as defined
in the statement of the present theorem.

According to \eqref{a8.3}, for some $c_5$ and any  $\varepsilon\in(0, \min(y_A-R,E/(Mg)-y_A))$, 
\begin{align*}
  \P_n(|Y_{n+1}-y_A|\geq \varepsilon) &\leq \frac{c_5}{Z_n''}\left(\int_R^{y_A-\varepsilon}+\int_{y_A+\varepsilon}^{E/(Mg)}\right)\psi^n(y)z^{d/2}(y)\rd y\/.
\end{align*}
Since
\begin{align*}
 Z_n''&\geq \int_{y_A}^{y_A+\varepsilon/2}\psi^n(y)z^{d/2}(y)\rd y
+\int_{y_A-\varepsilon/2}^{y_A}\psi^n(y)z^{d/2}(y)\rd y\\
 &\geq \frac{\varepsilon}{2}\psi^n(y_A+\varepsilon/2)z^{d/2}(y_A+\varepsilon/2)
 + \frac{\varepsilon}{2}\psi^n(y_A-\varepsilon/2)z^{d/2}(y_A)\/,
\end{align*}
and
\begin{align*}
& \left(\int_R^{y_A-\varepsilon}+\int_{y_A+\varepsilon}^{E/(mg)}\right)\psi^n(y)z^{d/2}(y)\rd y
\leq (E/(Mg) - R) z^{d/2}(R) (\psi^n(y_A-\varepsilon) + \psi^n(y_A+\varepsilon)),
\end{align*}
we have 
\begin{align*}
  \P_n(|Y_{n+1}-y_A|\geq \varepsilon)
&\leq
\frac
{(E/(Mg) - R) z^{d/2}(R) (\psi^n(y_A-\varepsilon) + \psi^n(y_A+\varepsilon))}
{\frac{\varepsilon}{2}\psi^n(y_A+\varepsilon/2)z^{d/2}(y_A+\varepsilon/2)
 + \frac{\varepsilon}{2}\psi^n(y_A-\varepsilon/2)z^{d/2}(y_A)}  \\
   &\leq \frac{2E}{\varepsilon Mg}\left(\frac{z(R)}{z(y_A+\varepsilon/2)}\right)^{d/2}\left[\left(\frac{\psi(y_A+\varepsilon)}{\psi(y_A+\varepsilon/2)}\right)^n +\left(\frac{\psi(y_A-\varepsilon)}{\psi(y_A-\varepsilon/2)}\right)^n\right].
\end{align*}
The right-hand side  goes to $0$ as $n\to \infty$, i.e.,
\begin{align}\label{a8.7}
\lim_{n\to\infty} \P_n(|Y_{n+1}-y_A|\geq \varepsilon)=0.
\end{align}

A calculation analogous to \eqref{a8.10} gives for $\delta >0$,
\begin{align}\label{a8.11}
&\left(\int_0^{u_0-\delta}+\int_{u_0+\delta}^{z(y)}\right)\alpha_y^{k}(u) (z(y)-u)^{d-1}\rd u\\
& \quad \leq\frac{2z^{d}(y)}{d\sqrt{k}}\left(\alpha^{k}_y\left(u_0-\delta\right)+\alpha^{k}_y\left(u_0+\delta\right)-\alpha_y^{k}(0)-\alpha_y^{k}(z(y))\right) \notag\\
&\quad\leq \frac{2z^{d}(y)}{d\sqrt{k}}\left(\alpha^{k}_y\left(u_0-\delta\right)+\alpha^{k}_y\left(u_0+\delta\right)\right).\notag
\end{align}
Recall that $k=n-1$ and, by \eqref{a8.8} and \eqref{a8.9}, that $\frac 1 n \sum_{i=1}^n y_i = u$. The following remark made above about $\psi$ applies also to $\alpha_y$ because of \eqref{a7.9}: ``$\psi(y)$ attains its maximum at $y_A$ and it is strictly increasing on $(R,y_A)$ and strictly decreasing on $(y_A,E/(Mg))$.''
Combining \eqref{a7.8} and \eqref{a8.11}, we obtain for every fixed $y$,
\begin{align}\label{a8.20}
\limsup_{n\to\infty} \P_n& \left(\left|\frac 1 n \sum_{i=1}^n Y_i - u_0(y)\right| >\delta \mid Y_{n+1} = y\right) \\
&\leq\limsup_{k\to\infty}
\frac
{\frac{2z^{d}(y)}{d\sqrt{k}}\left(\alpha^{k}_y\left(u_0-\delta\right)+\alpha^{k}_y\left(u_0+\delta\right)\right)}
{\frac{c_4}{\sqrt{k}}z^d(y)\alpha_y^k(u_0)}
=0.\notag
\end{align}
Using the fact that $u_0(y)$ is a continuous function of $y$ (see, e.g., \eqref{a8.5}),  \eqref{a8.7}, \eqref{a8.20} and applying the dominated convergence theorem  to the indicator function of the event $\left\{ \left|\frac 1 n \sum_{i=1}^n Y_i - u_0(y)\right| >\eps \right\}$ we obtain from \eqref{a8.7} and \eqref{a8.20},
\begin{align*}
\lim_{n\to\infty} \P_n \left(|Y_{n+1} - y_A|> \eps
\text{  or  } \left| \sum_{i=1}^n Y_i - u_0(y_A)\right| >\eps\right) =0.
\end{align*}
It remains to recall that we have shown that $u_0(y_A) = u_A$.
\end{proof}

Recall  definitions \eqref{j24.4} of $( \bX_n,\bY_n)$, \eqref{j24.3} of $\nu_{x,y,\lambda}$, and \eqref{j13.2} of $\bD_n(x,y,u)$.

\begin{lemma}\label{j15.1}
Fix any $\wh x\in D_b'$, $u>0$ and $\wh y\geq R$, and
let $\wh \cB= \cB(( \wh x, \wh y), R)$.
Let $\P_n^{\wh x,\wh y,u}$ denote the uniform distribution on $\bD_n(\wh x,\wh y,u)$. Suppose that $( \bX_n,\bY_n)$ has the distribution $\P_n^{\wh x,\wh y,u}$.

(i) 
When $n\to \infty$ then for every fixed $j\geq 1$, the distribution of $(X_j,Y_j)$ converges to $\nu_{\wh x,\wh y,\lambda}$, where $\lambda$ is the solution to $ u = \scru(\lambda,y)$. Moreover, for any $j_1\ne j_2$, the joint distribution of $((X_{j_1},Y_{j_1}),(X_{j_2}, Y_{j_2}))$ converges to $\nu_{\wh x,\wh y,\lambda} \times \nu_{\wh x,\wh y,\lambda}$.

(ii)
The convergence in (i) is uniform in the sense that for any bounded $d$-dimensional rectangular parallelepipeds $A_1,A_2\subset \R^d$ with non-empty interior, any  $0< u_1 < u_2 < \infty$, and any $\eps>0$, there exists $n_1$ such that for all $n\geq n_1$, $\wh x \in D_b'$, $y>R$ and $u\in[u_1,u_2]$,
\begin{align}\label{j13.6}
\left| \P_n^{\wh x,y,u} (((X_{j_1},Y_{j_1}),(X_{j_2}, Y_{j_2})) \in A_1 \times A_2) - \nu_{\wh x,y,\lambda} \times \nu_{\wh x,y,\lambda}
(A_1 \times A_2) \right| < \eps.
\end{align}

\end{lemma}

\begin{proof}
We will reuse some ideas from the proof of Proposition \ref{j11.2}.

Let $  (Z^{1 }_1,T^{1 }_1),\ldots,(Z^{1 }_{n},T^{1 }_n)$ be i.i.d., with $Z^{1 }_i \in \R^{d-1}$ and $T^{1 }_i >0$ for $1\leq i \leq n$. Assume that $T^{1 }_i$ and $Z^{1 }_i$ are independent, $Z^{1 }_i$ has the uniform distribution in $D_b$, and $T^{1 }_i$  has the exponential distribution with parameter $\lambda$.

Let $f_1(z_1,t_1, \dots,z_n, t_n)$ be the density of $\left((Z^{1 }_1,T^{1 }_1),\ldots,(Z^{1 }_{n},T^{1 }_n)\right)$.
We have 
\begin{align}\label{j11.3}
f_1(z_1,t_1, \dots,z_n, t_n)
= |D_b|^{-n} \lambda ^n \exp(-\lambda(t_1+ \dots + t_n)),
\end{align}
for $z_i \in D_b$, $t_i >0$, $1\leq i \leq n$. 

Let $(Z^{2 }_1,T^{2 }_1),\ldots,(Z^{2 }_{n},T^{2 }_n)$ be i.i.d., with $( Z^{2 }_i,  T^{2 }_i)$ being distributed as $( Z^{1 }_i,  T^{1 }_i)$ conditioned by 
$\{(Z^{1 }_i,T^{1 }_i)\notin \wh \cB\}$.
Hence, $(Z^{2 }_i,T^{2 }_i)$ has the distribution $\nu_{\wh x, \wh y, \lambda}$.

 By the definition of $q$ and \eqref{j11.3}, the density $f_2$ of $((Z^{2 }_1,T^{2 }_1),\ldots,(Z^{2 }_{n},T^{2 }_n))$ is given by
\begin{align}\label{j11.4}
f_2(z_1,t_1, \dots,z_n, t_n)
= \left(\frac q {|D_b|}\right)^{-n} |D_b|^{-n} \lambda^n \exp(-\lambda(t_1+ \dots + t_n))
\end{align}
on the set $(D\setminus \wh \cB)^n$.

Let $S^2_n = \frac 1 n \sum_{k=1}^n T^2_k $. 
Let $((Z^{3 }_1,T^{3 }_1),\ldots,(Z^{3 }_{n},T^{3 }_n))$ be
the sequence  $((Z^{2 }_1,T^{2 }_1),\ldots,(Z^{2 }_{n},T^{2 }_n))$ 
conditioned by $\{S^2_n = u\}$. Note that the distribution of $((Z^{3 }_1,T^{3 }_1),\ldots,(Z^{3 }_{n},T^{3 }_n))$ is the uniform distribution on $\bD_n(\wh x,\wh y,u)$. Hence, the distribution of $(X_j,Y_j)$ under $\P_n^{\wh x,\wh y,u}$ is the same as the distribution of $(Z^{3 }_{j},T^{3 }_j)$. 
We will show that the distribution of $(Z^{3 }_{j},T^{3 }_j)$ converges to that of $(Z^{2 }_{j},T^{2 }_j)$ as $n\to \infty$. This is equivalent to the weak convergence of $\P_n^{\wh x,\wh y,u}$ to $\nu_{\wh x, \wh y, \lambda}$ since the distribution of $(Z^{2 }_i,T^{2 }_i)$ has the distribution $\nu_{\wh x, \wh y, \lambda}$.

Fix $j\geq 1$ and consider $n> j$. Let
$S^{2,j}_n = \frac 1 {n-1} \sum_{k=1, \dots, n; k\ne j} T^2_k $.

Since  $(Z^{2 }_1,T^{2 }_1),\ldots,(Z^{2 }_{n},T^{2 }_n)$ is an i.i.d. sequence, the density $f_{(Z^{3 }_{j},T^{3 }_j)}$ has the form
\begin{align}\label{j11.5}
f_{(Z^{3 }_{j},T^{3 }_j)}(z,t)
= c_1 f_{(Z^{2 }_{j},T^{2 }_j)}(z,t)
f_{S^{2,j}_n} \left(u+ \frac 1 {n-1} (u-t)\right),
\end{align}
where $c_1$ is the normalizing constant.

Let $\sigma^2 = \var T^2_i$ and $T^4_i = (T^2_i - u)/(\sigma\sqrt{n-1})$. Note that, since $ u = \scru(\lambda,y)$, $\E T^4_i=0$. Let
\begin{align*}
S^{4,j}_n = \sum_{k=1, \dots, n; k\ne j} T^4_k
= \frac 1 {\sigma\sqrt{n-1}} \sum_{k=1, \dots, n; k\ne j} (T^2_k - u)
= \frac{\sqrt{ n-1}}{\sigma} (S^{2,j}_n-u).
\end{align*}
Hence 
\begin{align*}
f_{S^{2,j}_n} (u+s) = \frac{\sqrt{ n-1}}{\sigma}  f_{S^{4,j}_n}\left(\frac{\sqrt{ n-1}}{\sigma}  s\right).
\end{align*}
By Lemma \ref{j10.2},
\begin{align}\label{j24.5}
 f_{S^{2,j}_n}&\left(u+ \frac 1 {n-1} (u-t)\right)
 = \frac{\sqrt{ n-1}}{\sigma}  f_{S^{4,j}_n}\left(\frac{\sqrt{ n-1}}{\sigma}  \frac 1 {n-1} (u-t)\right)\\
 &= \frac{\sqrt{ n-1}}{\sigma}  f_{S^{4,j}_n}
 \left(\frac 1 {\sigma\sqrt{n-1}}  (u-t)\right)
  =\frac{\sqrt{ n-1}}{\sigma}
 \left(\phi\left( \frac 1 {\sigma\sqrt{n-1}} (u-t)\right) + A\right) .\notag
\end{align}
Since the random variable $T^2_i $ has the same distribution as the random variable with the same name in the proof of Proposition \ref{j11.2}, the estimates 
\eqref{y22.10}-\eqref{y22.11} for $A$, and those at the end of the proof of 
Proposition \ref{j11.2} apply in the present case. Thus 
\begin{align}\label{j26.1}
|A| \leq \frac {c_2}{\sqrt{n-1}},
\end{align}
where $c_2$ depends only on $|D_b|$.

Fix any $u_1 >0$.
Then, by \eqref{y22.6} and \eqref{y22.2}, there exists $c_3>0$ such that for all $u \geq u_1$, 
\begin{align}\label{j26.2}
\sigma^2 \geq c_3.
\end{align}

Fix any $u_2\in(u_1,\infty)$. It follows from \eqref{j24.5}, \eqref{j26.1} and \eqref{j26.2} that for any fixed $0 \leq t_1 < t_2 < \infty$, $u\in[u_1,u_2]$ and $\eps>0$ there exists $n_1$ such that for $n\geq n_1$, $t_3,t_4\in[t_1,t_2]$, and $t_5\notin[t_1,t_2]$,
\begin{align}\label{j11.6}
1-\eps <\ 
&\frac{
f_{S^{2,j}_n}\left(u+ \frac 1 {n-1} (u-t_3)\right)}
{f_{S^{2,j}_n}\left(u\right)}
<1+\eps,\\
&\frac{
f_{S^{2,j}_n}\left(u+ \frac 1 {n-1} (u-t_5)\right)}
{f_{S^{2,j}_n}(u)}
<2.\label{j11.7}
\end{align}
This, \eqref{y22.8} and \eqref{j11.5} imply that the distribution of $(Z^{3 }_{j},T^{3 }_j)$ converges to that of $(Z^{2 }_{j},T^{2 }_j)$ as $n\to \infty$. 
Moreover,
for any bounded $d$-dimensional rectangular parallelepiped $A_1\subset \R^d$ with non-empty interior, any $0< u_1 < u_2 < \infty$, and any $\eps>0$, there exists $n_1$ such that for all $n\geq n_1$, $\wh x \in D_b'$, $y>R$ and $u\in[u_1,u_2]$,
\begin{align*}%\label{j13.6}
\left| \P_n^{\wh x,y,u} ((X_{j_1},Y_{j_1}) \in A_1) - \nu_{\wh x,y,\lambda} 
(A_1 ) \right| < \eps.
\end{align*}

A completely analogous argument shows that for any $j_1\ne j_2$, 
the distribution of $\left((Z^{3 }_{j_1},T^{3 }_{j_1}),(Z^{3 }_{j_2},T^{3 }_{j_2})\right)$ converges to that of $\left((Z^{2 }_{j_1},T^{2 }_{j_1}),(Z^{2 }_{j_2},T^{2 }_{j_2})\right)$ as $n\to \infty$, and also part (ii) of the lemma holds true. 
\end{proof}

Recall 
 $\nu_{\wx,y,\lambda}$  defined in \eqref{j24.3},   the empirical measure $\Q_n$ defined in \eqref{j13.1},
and  $\bD_n(x,y,u)$ defined in \eqref{j13.2}.

\begin{lemma}\label{j14.1}
The marginal distribution of $X_{n+1}$ under $\P_n$ is uniform in $D_b'$. Given $\{X_{n+1}=x\}$, the conditional distribution of  $\Q_n$ converges to $\nu_{x,\ya,\lambda_A}$ weakly, in probability as $n\to \infty$.
\end{lemma}

\begin{proof}
It follows easily from the microcanonical ensemble formula \eqref{j10.4} that the marginal distribution of $X_{n+1}$ under $\P_n$ is uniform in $D_b'$.

The same formula \eqref{j10.4} implies that
one can represent $\P_n$ as follows. Let $\P_n^{x,y,u}$ denote the uniform distribution on $\bD_n(x,y,u)$. Then there exists  a probability measure $\mu_n$  on $D_b'\times [R,\infty) \times \R_+$ such that  $\P_n = \int \P_n^{x,y,u} \rd \mu_n(x,y,u)$. 
Let $\P_n^x = \int \P_n^{x,y,u} \mu_n(x,\rd y,\rd u)$.
Let $\E_n^x$ and $\E_n^{x,y,u} $ denote expectations corresponding to $\P_n^x$ and $\P_n^{x,y,u}$.

Fix any $x\in D_b'$ and write $\nu_x'= \nu_{x,y_A,\lambda_A}$.

Fix any bounded $d$-dimensional rectangular parallelepiped $A\subset \R^d$ with non-empty interior. 
We have
\begin{align}\label{j13.4}
&\E_n^x \left( \Q_n(A) - \nu_x'(A)\right)^2
= \E_n^x \left( \frac 1 n \sum_{i=1}^n \bone_{A} (X_i,Y_i) - \nu_x'(A)\right)^2\\
&= \E_n^x \left( \frac 1 n \sum_{i=1}^n 
\left(\bone_{A} (X_i,Y_i) - \nu_x'(A)\right)\right)^2 \notag\\
&= \frac 1 {n^2} \sum_{i=1}^n \E_n^x
\left(\bone_{A} (X_i,Y_i) - \nu_x'(A)\right)^2 \notag\\
&\qquad+ \frac 2 {n^2} 
 \sum_{i=1}^{n-1} \sum_{j=i+1}^n \E_n^x \left[
\left(\bone_{A} (X_i,Y_i) - \nu_x'(A)\right) 
\left(\bone_{A} (X_j,Y_j) - \nu_x'(A)\right) \right] \notag\\
&= \frac 1 n \E_n^x
\left(\bone_{A} (X_1,Y_1) - \nu_x'(A)\right)^2  \notag\\
&\qquad+ \frac {2(n^2-n)}{n^2} 
\E_n^x \left[
\left(\bone_{A} (X_1,Y_1) - \nu_x'(A)\right) 
\left(\bone_{A} (X_2,Y_2) - \nu_x'(A)\right) \right] \notag\\
&= \frac 1 n \E_n^x
\left(\bone_{A} (X_1,Y_1) - \nu_x'(A)\right)^2  \notag\\
&\qquad+ \frac {2(n^2-n)}{n^2} \Big(
\E_n^x \left[
\bone_{A} (X_1,Y_1)\bone_{A} (X_2,Y_2)\right]
-\E_n^x \left[
\bone_{A} (X_1,Y_1)\nu_x'(A)\right] \notag\\
&\qquad \qquad-\E_n^x \left[
\nu_x'(A)\bone_{A} (X_2,Y_2)\right]
+
(\nu_x'(A))^2\Big).\notag
\end{align}

It is easy to see that
the function $(x,y,u) \to \P_n^{x,y,u}$ is continuous in the weak topology.
This, Proposition \ref{j13.3} and \eqref{j13.6} imply that
\begin{align}\label{j13.7}
&\lim_{n\to \infty} \E_n^x \left[
\bone_{A} (X_1,Y_1)\bone_{A} (X_2,Y_2)\right] \\
&\quad= \lim_{n\to \infty}\int \E_n^{x,y,u} \left[
\bone_{A} (X_1,Y_1)\bone_{A} (X_2,Y_2)\right]   \mu_n(x,\rd y,\rd u)
= (\nu_x'(A))^2.\notag
\end{align}
For the same reason,
\begin{align}\label{j13.8}
\lim_{n\to \infty} \E_n^x \left[
\bone_{A} (X_k,Y_k)\right] &= \nu_x'(A), \qquad k=1,2.
\end{align}
The bound $\E_n^x
\left(\bone_{A} (X_1,Y_1) - \nu_x'(A)\right)^2 \leq 1$, \eqref{j13.4}, \eqref{j13.7} and \eqref{j13.8}
imply that 
\begin{align*}
\lim_{n\to \infty}\E_n^x \left( \Q_n(A) - \nu_x'(A)\right)^2 =0.
\end{align*}
The lemma follows because this statement holds for every $A$.
\end{proof}

\begin{proof}[Proof of Theorem \ref{a20.4}]
The theorem follows from Proposition \ref{j13.3} and  Lemma \ref{j14.1}.
\end{proof}

\section{Inverse temperature}\label{y19.7}

\begin{proposition}
The inverse temperature of the gas is asymptotically proportional to $\lambda_A$
defined in Theorem \eqref{j24.1}. More precisely, for every fixed $j\geq 1$,
\begin{align} 
\lim_{n\to\infty} \E_n \left(\frac 1 2 m \|V_j\|^2\right)=
\lim_{n\to\infty} \E_n \left(\frac 1 n\sum_{i=1}^{n}\frac 1 2 m \|V_i\|^2\right)&=\frac{dmg}{2\lambda_A}.\label{j15.5}
\end{align}
\end{proposition}

\begin{proof}
Since the total energy $E$ is fixed, we have
\begin{align*}%\label{a20.2}
\frac 1 n\sum_{i=1}^{n} mg Y_i \leq E, 
\qquad MgY_{n+1} \leq E,
\qquad \frac 1 {n+1}\sum_{i=1}^{n+1}\frac 1 2 m \|V_i\|^2\leq E.
\end{align*}
Hence, these sequences of random variables are tight. 
Recall $\P_n^{x,y,u}$ and $\mu_n$
from the proof of Lemma \ref{j14.1}.
By Proposition \ref{j13.3},
\begin{align}\label{a8.14}
\lim_{n\to\infty} \E_n &\left(\frac 1 n\sum_{i=1}^{n} mg Y_i\right)
=
\lim_{n\to\infty}\int \E_n^{x,y,u} \left(\frac 1 n\sum_{i=1}^{n} mg Y_i\right)\rd \mu_n(x,y,u)\\
&=
\lim_{n\to\infty}\int u \rd \mu_n(x,y,u)
= mg u_A
=mg \frac
{ \int_{D\setminus\cB((\wx, y_A), R)} y\lambda_A e^{-\lambda_A y} \rd x \rd y} 
{ \int_{D\setminus\cB((\wx, y_A), R)} \lambda_A e^{-\lambda_A y} \rd x \rd y}.
\notag
\end{align}
Proposition \ref{j13.3} and tightness also imply that
\begin{align}\label{a8.13}
\lim_{n\to\infty} \E_n MgY_{n+1} = Mgy_A.
\end{align}
This and \eqref{a8.14} show that the expected value of the potential energy of the point particles and the ball converges to a fixed number.
 Since the total energy $E$ is fixed, the expectation of the kinetic energy converges weakly to a fixed number as well, i.e., for some $\sigma>0$,
\begin{align}\label{j15.3}
\lim_{n\to\infty} \E_n\left(\frac 1 {n+1}\sum_{i=1}^{n+1}\frac 1 2 m \|V_i\|^2\right) =  \sigma^2.
\end{align}
The total energy is fixed so
 \eqref{a8.14}, \eqref{a8.13} and \eqref{j15.3} show that
\begin{align*}
\sigma^2 + 
mg \frac
{ \int_{D\setminus\cB((\wx, y_A), R)} y\lambda_A e^{-\lambda_A y} \rd x \rd y} 
{ \int_{D\setminus\cB((\wx, y_A), R)} \lambda_A e^{-\lambda_A y} \rd x \rd y} 
+ Mg y_A = E.
\end{align*}
Since $y_A$ and $\lambda_A$ solve \eqref{j27.10}, we must have $\sigma^2 = \frac{dmg}{2\lambda_A}$.

It remains to note that due to the symmetry of $\mu_{\by_{n+1}}$ in \eqref{j10.4},
\begin{align*}
\lim_{n\to\infty} \E_n \left(\frac 1 n\sum_{i=1}^{n}\frac 1 2 m \|V_i\|^2\right)=
\lim_{n\to\infty} \E_n\left(\frac 1 {n+1}\sum_{i=1}^{n+1}\frac 1 2 m \|V_i\|^2\right) 
=  \sigma^2= \frac{dmg}{2\lambda_A}.
\end{align*}
\end{proof}

\section{Uniqueness of the stationary distribution}\label{y19.8}

\begin{proof}[Proof of Theorem \ref{j18.1} (ii)]
The idea of the proof is the following. If there were more
than one invariant measure, at least two of them  would be mutually
singular by Birkhoff's ergodic theorem (\cite{Sin94}). Given any two starting configurations we will exhibit two deterministic trajectories meeting at the same point in the phase space at some time $t_1>0$. Then we will argue that due to the random nature of some reflections, both processes have densities that are strictly positive in some neighborhood of that point in the phase space.
Hence,  there are no mutually singular invariant measures.

Much of the proof will be presented in a very informal way. This is because our argument is totally elementary but it would be extremely tedious to write (or read) in a fully rigorous way.

\medskip
{\it Step 1}.
Assume that the initial condition of the system does not belong to any of the families (1) and (2) described in the statement of the theorem.
We will construct a single trajectory of the system. The trajectory will respect the laws of elastic collisions when they are assumed, i.e., for all collisions of the ball with the point particles and the walls of the container. For any reflection of a point particle from a wall of the container, we will choose a  direction after the reflection from all possible directions in a way that meets the needs of the argument. 

Recall that ``walls'' of the container include its bottom so point particles reflect according to the Lambertian law from the side walls and the bottom of the container.

Fix distinct points $z_1, \dots, z_{n+1}$ in $D_b$, such that the distance of $z_{n+1}$ from the side wall of the container is greater than $R$. 

We let the system evolve  according to the original dynamics until one of the point particles hits a wall at a time $s_1$. We  assume that the particle that hit the wall is labeled 1, since the labeling of point particles is irrelevant.

The first point particle can reflect in any direction, including directions arbitrarily close to the boundary. So  it can stay arbitrarily close to the boundary for an arbitrarily long time and move towards any point on the boundary, with the only limitation being its constant energy (the sum of potential and kinetic energies), see Figs. \ref{fig1}-\ref{fig2}. We let the system evolve  according to the original dynamics after time $s_1$, except that the first point particle will follow its own trajectory, constructed independently. Let $s_2$ be the next time when a particle different from the first one hits a wall.  We choose a trajectory  for the first particle  very close to the wall and moving towards $z_1$ in such a way that it avoids a collision with the ball on $[s_1,s_2]$ (recall that point particles do not interact).

\begin{figure} \includegraphics[width=0.6\linewidth]{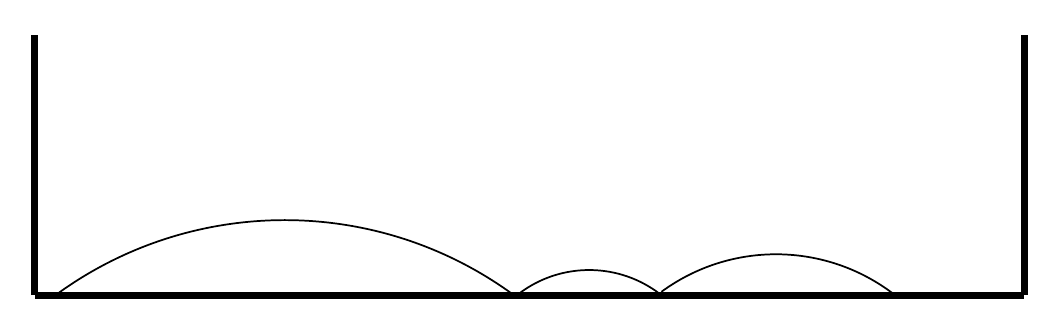}
\caption{Side view of a point particle trajectory reflecting from the bottom of the container. 
}
\label{fig1}
\end{figure}

\begin{figure} \includegraphics[width=0.6\linewidth]{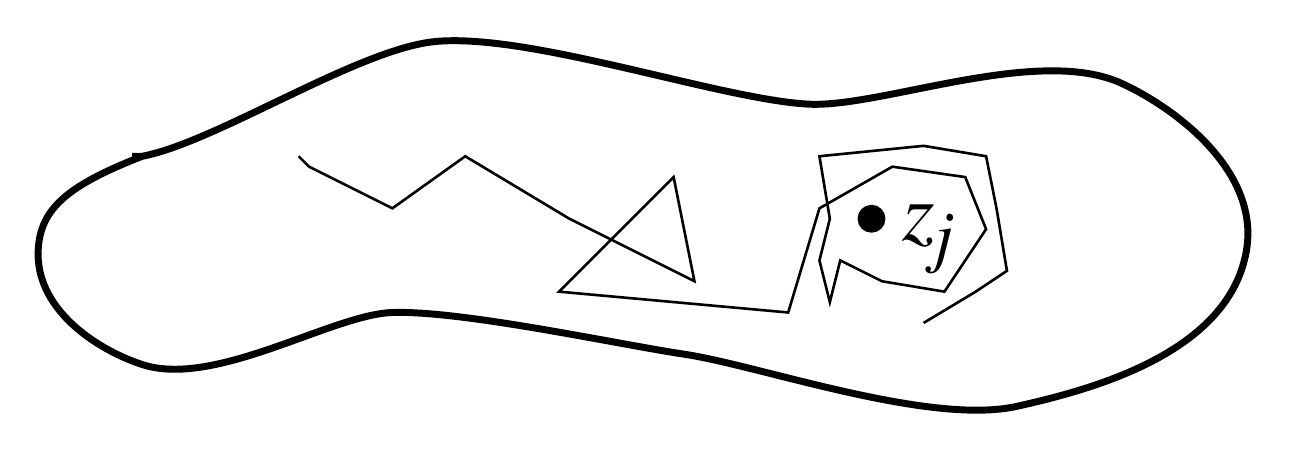}
\caption{View from above of a point particle trajectory reflecting from the bottom of the container. 
}
\label{fig2}
\end{figure}

We proceed by induction. Suppose that, for some $j<n$, a deterministic trajectory of the system has been constructed on $[0,s_j]$, including the trajectories of point particles  labeled $1,\dots , j$. These point particles stay close to the walls from the time of the first hit of a wall until $s_j$. Particle $k$ moves towards $z_k$ from the first time it hits a wall until $s_j$, for $k=1,\dots,j$. Given this inductive assumption, we let the system evolve  according to the original dynamics after time $s_j$, except that  point particles labeled $1,\dots, j$ will follow their own trajectories, constructed independently. Let $s_{j+1}$ be the next time when a particle different from $1,\dots,j$ hits a wall. We will call this particle $j+1$.  We choose a trajectory  for each of the  particles $1,\dots, j+1$  very close to the wall. The $k$-th particle is moving towards $z_k$ in such a way that it avoids a collision with the ball on $[s_j,s_{j+1}]$, for $k=1,\dots, j+1$.

We stop the construction when we have a trajectory of the system on $[0,s_{n}]$. We will continue after discussing a delicate point in the next step.

\medskip
{\it Step 2}. It is possible that fewer than $n$ point particles hit the walls of the container. This can happen only if some particles always reflect from the top part of the ball.
In this case, we let one of the point particles that are staying close to the boundary move towards the point where the ball reflects from the bottom of the container and then we let the point particle hit the ball slightly off center. That will nudge the ball off its  trajectory. The result will be that the point particles formerly reflecting from the top of the ball will move to the side and eventually hit a wall.

\medskip
{\it Step 3}.
At this step of the construction of the trajectory of the system, point particles will come close to the point at the bottom where the ball reflects and they will hit the ball, one at a time, see Fig. \ref{fig3}. All other point particles will keep close to the bottom and stay away from the ball.

\begin{figure} \includegraphics[width=0.6\linewidth]{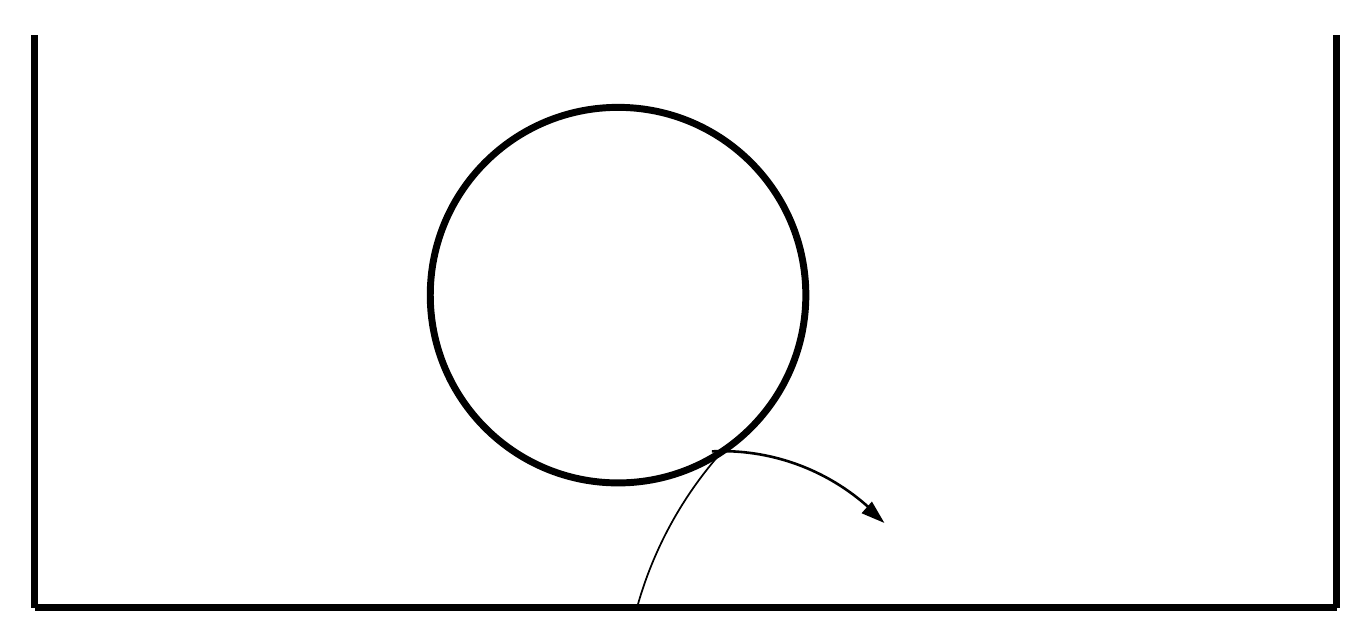}
\caption{Side view of the container. A point particle collides with the ball. 
}
\label{fig3}
\end{figure}

There are two major goals of the construction that can be achieved by this procedure. First, we can achieve equipartition of the energy. Second, we can put the ball above $z_{n+1}$ and make it move vertically.

We will explain how we can arrange for equipartition of energy between the point particles and the ball, so that $(E-MgR)/(n+1)$ of energy is given to each point particle and $(E-MgR)/(n+1) + MgR$ of energy is given to the ball.
Note that 
the minimal amount of energy a point particle can have is 0, assuming that it is sitting motionless at the bottom. For the ball, the minimal amount of energy is $MgR$.

The point particles do not interact with each other so the only way to transfer the energy between them is via collisions with the macroscopic ball.
We let
point particles approach the point where the ball reflects from the bottom, one at a time. Then we make the direction of the velocity of the point particle close to vertical (or at least considerably different from the horizontal; see Fig. \ref{fig3}). We make the point particle hit the ball either when the ball is moving up or down. In the first case, the point particle will lose energy and in the second case it will  acquire energy. By manipulating the place of the collision and the velocity direction of the point particle before the collision, and by repeating the procedure, if necessary, many times, we can partition the energy between particles and the ball in an arbitrary way. 

We need to add a few words clarifying the algorithm described above.
If a particle and the ball have the same amount of kinetic energy and $n$ is large then the speed of the particle is much larger than the speed of the ball. Hence, we can start by transferring energy to the ball  
from all point particles that have more than $(E-MgR)/(n+1)$ of  energy. Then the energy can be transferred from the ball  to the particles that had less than 
$(E-MgR)/(n+1)$ of  energy, one by one.

\medskip
{\it Step 4}.
We make the $n$-th point particle  collide with the ball as depicted in Fig. \ref{fig3} to change the trajectory of  the ball so it reflects vertically at $z_{n+1}\in D_b$.

After this is done, energy might not be equidistributed.
If necessary, we induce  energy transfer between the $n$-th particle and the ball by collisions of the particle with the bottom of the ball, as in Fig. \ref{fig4}.

\begin{figure} \includegraphics[width=0.6\linewidth]{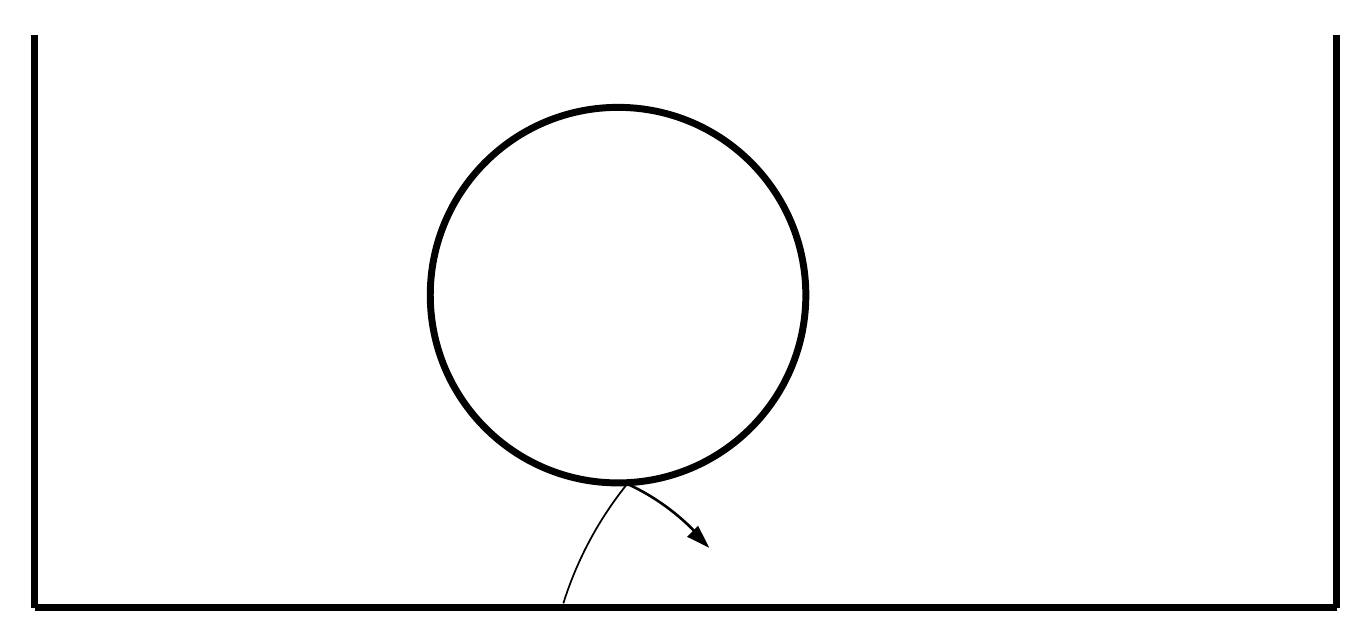}
\caption{Side view of the container. A point particle hits the ball at the lowest point on the surface. 
}
\label{fig4}
\end{figure}

Fix some time $t_1$ greater than the duration of the trajectory constructed so far, such that the ball hits $z_{n+1}$ at time $t_1$.
Make all point particles follow  trajectories such that  they all hit their own base points $z_k$, $k=1,\dots, n$, at time $t_1$. For this to be possible, it may be necessary to move $t_1$ to one of the later times when the ball hits $z_{n+1}$.

For future reference, let the deterministic trajectory constructed above be called $\Gamma = \{\Gamma(t), 0\leq t \leq t_1\}$.

\medskip
{\it Step 5}.
To finish the proof, we will argue as follows. We have shown that the system  can get to the same configuration at time $t_1$ for every initial configuration (the time $t_1$ may depend on the initial conditions---this is not a problem). The trajectory that we constructed is deterministic but it ``agrees'' with the dynamics of the system, including Lambertian reflections. Lambertian reflections introduce randomness. They 
make the state of the system at time $t_1$ random, with a density. The densities for different initial configurations overlap so there is only one stationary measure.

The subtle point is that the state density at time $t_1$ is not with respect to Lebesgue measure on $\R^{2nd}$ but on a hypersurface of dimension $2nd-1$ because the total energy is fixed (see \eqref{y19.1}).
We will now outline an argument addressing this concern.

Recall notation from Section \ref{y19.2} and let
\begin{align*}
\bZ_j(t) &= ( (X_1(t), \dots, X_j(t)), (Y_1(t), \dots, Y_j(t)), (V_1(t), \dots, V_j(t))) ,\\
\bZ_{j+}(t) &= ( (X_j(t), \dots, X_{n+1}(t)), (Y_j(t), \dots, Y_{n+1}(t)), (V_j(t), \dots, V_{n+1}(t))) ,\\
\bz_j &= ( (x_1, \dots, x_j), (y_1, \dots, y_j), (v_1, \dots, v_j)) ,\\
\bz_{j+} &= ( (x_j, \dots, x_{n+1}), (y_j, \dots, y_{n+1}), (v_j, \dots, v_{n+1})) ,
\end{align*}
where $x_k\in \R^{d-1}$, $y_k\in R$ and $v_k \in \R^d$.

For arbitrarily thin tube around $\Gamma$ there is a strictly positive probability that the system with random reflections $\bZ_{n+1}$ will stay inside the tube until time $t_1$. Consider a tube so thin that point particles undergoing random reflections in the tube collide with the ball 
in the same order as along the deterministic trajectory $\Gamma$.

We will assume without loss of generality that the point particles exchange energy with the ball along $\Gamma$ in the order $1,\dots, n$. 

Let $u_1$ be the last time the first particle hits a wall of the container before starting the process of exchanging the energy with the ball. 
If  the tube is very thin, the first particle will not collide with the ball after exchanging the energy with the ball. 
Let $u_1'$ be the last time  the first ball collides with the ball before time $t_1$.
We claim that for some small neighborhood $U_1$ of $\Gamma (t_1)$, some $c_1>0$ and $\bz_{n+1}=(\bz_1, \bz_{2+}) \in U_1$,  
\begin{align}\label{a1.2}
&\P_n(\bZ_1(t_1)\in \rd \bz_1 \mid \bZ_{n+1}(u_1))/\rd \bz_1 \geq c_1,\\
&\P_n(\bZ_{n+1}(t_1)\in \rd \bz_{n+1} \mid \bZ_{n+1}(u_1'))/\rd \bz_{n+1}  \geq c_1
\P_n(\bZ_{2+}(t_1)\in \rd \bz_{2+} \mid \bZ_{n+1}(u_1'))/\rd \bz_{2+}.\label{a1.3}
\end{align}
The  claim \eqref{a1.2} is true because the first particle acquires a random amount of energy, and its position and velocity direction are random due to Lambertian reflections following $u_1$. The  claim \eqref{a1.3} is a form of independence for conditional processes (conditioned to stay in separate tubes).

We proceed by induction. Suppose $2\leq j \leq n-1$.
Let $u_j$ be the last time the $j$-th particle hit a wall of the container before starting the process of exchanging the energy with the ball. 
Let $u_j'$ be the last time  the $j$-th particle collided with the ball before time $t_1$. Then for some small neighborhood $U_j\subset U_{j-1}$ of $\Gamma (t_1)$, some $c_j>0$ and $\bz_{n+1}=(\bz_j, \bz_{j+}) \in U_j$,  
\begin{align*}
&\P_n(\bZ_j(t_1)\in \rd \bz_j \mid \bZ_{n+1}(u_j))/\rd \bz_j \geq c_j,\\
&\P_n(\bZ_{n+1}(t_1)\in \rd \bz_{n+1} \mid \bZ_{n+1}(u_j'))/\rd \bz_{n+1}  \geq c_j
\P_n(\bZ_{(j+1)+}(t_1)\in \rd \bz_{(j+1)+} \mid \bZ_{n+1}(u_j'))/\rd \bz_{(j+1)+}.
\end{align*}
We apply this claim to $j=n-1$ to obtain
\begin{align*}
\P_n(\bZ_{n+1}(t_1)\in \rd \bz_{n+1} \mid \bZ_{n+1}(u_{n-1}'))/\rd \bz_{n+1}  \geq c_1
\P_n(\bZ_{n+}(t_1)\in \rd \bz_{n+} \mid \bZ_{n+1}(u_{n}'))/\rd \bz_{n+}.
\end{align*}
At this point it remains to analyze the interaction of the $n$-th particle and the ball. The position of the ball and its energy and velocity direction can be all made to have a joint density  by collisions with the $n$-th point particle after time $u_n$.

Finally, the position and velocity direction of the $n$-th point particle have a conditional density  given everything else, because of its Lambertian reflections from the bottom of the container. The energy of the $n$-th particle cannot be adjusted but this is fine because the same energy conservation principle applies to systems starting from other initial conditions.
\end{proof}

\begin{remark}\label{a1.1}
(i) There are many elementary examples of invariant measures for our dynamical system if we assume specular reflections of point particles from the walls of the container. To construct one of them, place all point particles  on disjoint vertical lines which do not intersect the macroscopic ball at time 0. Make initial velocities vertical for the ball and all point particles. Each of these objects will stay on a vertical line forever. To construct an invariant measure, use the ergodic theorem. 

(ii)
The following  example of an invariant distribution illustrates family (1) in Theorem \ref{j18.1}. Suppose that the macroscopic ball has a vertical initial velocity and all point particles are located on the vertical line passing through the ball's center, they are placed above the ball, and they all  have vertical initial velocities. In this case the center of the ball and the point particles will remain on the same vertical line forever and their velocities will also remain vertical (although the positions and velocities will not remain constant). For this initial condition, point particles will not hit the walls of the container so there will be no opportunity for the random reflections to cause mixing in the system.
Just like in part (i) of the remark, one can use the ergodic theorem to construct an invariant measure. 
\end{remark}

\section{Acknowledgments}

We are grateful to Shuntao Chen, Persi Diaconis, Martin Hairer, Robert Ho\l yst, Werner Krauth, Mathew Penrose and David Ruelle for the most useful advice.

\bibliographystyle{alpha}
\bibliography{arch_bib}

\newcommand{\etalchar}[1]{$^{#1}$}
\begin{thebibliography}{BDG{\etalchar{+}}21}

\bibitem[ABS13]{ABS}
Omer Angel, Krzysztof Burdzy, and Scott Sheffield.
\newblock Deterministic approximations of random reflectors.
\newblock {\em Trans. Amer. Math. Soc.}, 365(12):6367--6383, 2013.

\bibitem[BCP11]{archim}
Krzysztof Burdzy, Zhen-Qing Chen, and Soumik Pal.
\newblock Archimedes' principle for {B}rownian liquid.
\newblock {\em Ann. Appl. Probab.}, 21(6):2053--2074, 2011.

\bibitem[BDG{\etalchar{+}}21]{fermi}
Krzysztof Burdzy, Mauricio Duarte, Carl-Erik Gauthier, Robin Graham, and Jaime
  San~Martin.
\newblock Fermi acceleration in rotating drums.
\newblock 2021.
\newblock (in preparation).

\bibitem[CL02]{CL2002}
N.~Chernov and J.~L. Lebowitz.
\newblock Dynamics of a massive piston in an ideal gas: oscillatory motion and
  approach to equilibrium.
\newblock {\em J. Statist. Phys.}, 109(3-4):507--527, 2002.
\newblock Special issue dedicated to J. Robert Dorfman on the occasion of his
  sixty-fifth birthday.

\bibitem[CLS02]{CLS2002}
N.~Chernov, J.~L. Lebowitz, and Ya. Sinai.
\newblock Scaling dynamics of a massive piston in a cube filled with ideal gas:
  exact results.
\newblock {\em J. Statist. Phys.}, 109(3-4):529--548, 2002.
\newblock Special issue dedicated to J. Robert Dorfman on the occasion of his
  sixty-fifth birthday.

\bibitem[Col78]{Collins}
George~W. Collins.
\newblock {\em The virial theorem in stellar astrophysics}.
\newblock Pachart Pub. House, Tucson, 1978.
\newblock Astronomy and astrophysics series ; v. 7.

\bibitem[Gor11]{Gorel}
Igor Gorelyshev.
\newblock On the dynamics in the one-dimensional piston problem.
\newblock {\em Nonlinearity}, 24(8):2119--2142, 2011.

\bibitem[IS15]{IS}
Masato Itami and Shin-ichi Sasa.
\newblock Nonequilibrium statistical mechanics for adiabatic piston problem.
\newblock {\em J. Stat. Phys.}, 158(1):37--56, 2015.

\bibitem[Knu34]{K1934}
Martin Knudsen.
\newblock {\em The Kinetic Theory of Gases: Some Modern Aspects}.
\newblock Methuen \& Co., London, 1934.
\newblock (Methuen's Monographs on Physical Subjects).

\bibitem[Lam60]{L1760}
J.H. Lambert.
\newblock {\em Photometria sive de mensure de gratibus luminis, colorum
  umbrae}.
\newblock Eberhard Klett, 1760.

\bibitem[Lie99]{Lieb99}
Elliott~H. Lieb.
\newblock Some problems in statistical mechanics that {I} would like to see
  solved.
\newblock {\em Phys. A}, 263(1-4):491--499, 1999.
\newblock STATPHYS 20 (Paris, 1998).

\bibitem[LPS00]{LPS}
J.~L. Lebowitz, J.~Piasecki, and Ya. Sinai.
\newblock Scaling dynamics of a massive piston in an ideal gas.
\newblock In {\em Hard ball systems and the {L}orentz gas}, volume 101 of {\em
  Encyclopaedia Math. Sci.}, pages 217--227. Springer, Berlin, 2000.

\bibitem[LSC02]{LSC2002}
L.~Lebovits, Ya. Sina\u{\i}, and N.~Chernov.
\newblock Dynamics of a massive piston immersed in an ideal gas.
\newblock {\em Uspekhi Mat. Nauk}, 57(6(348)):3--86, 2002.

\bibitem[MM77]{Mayer}
Joseph~Edward Mayer and Maria~Goeppert Mayer.
\newblock {\em Statistical mechanics}.
\newblock John Wiley \& Sons, New York-London-Sydney, second edition, 1977.

\bibitem[NS04]{NeiSin}
A.~I. Neishtadt and Y.~G. Sinai.
\newblock Adiabatic piston as a dynamical system.
\newblock {\em J. Statist. Phys.}, 116(1-4):815--820, 2004.

\bibitem[Pet75]{Petrov}
V.~V. Petrov.
\newblock {\em Sums of independent random variables}.
\newblock Springer-Verlag, New York-Heidelberg, 1975.
\newblock Translated from the Russian by A. A. Brown, Ergebnisse der Mathematik
  und ihrer Grenzgebiete, Band 82.

\bibitem[Pla12]{Plakh}
Alexander Plakhov.
\newblock {\em Exterior billiards}.
\newblock Springer, New York, 2012.
\newblock Systems with impacts outside bounded domains.

\bibitem[Rue99]{Ruelle}
David Ruelle.
\newblock {\em Statistical mechanics}.
\newblock World Scientific Publishing Co., Inc., River Edge, NJ; Imperial
  College Press, London, 1999.
\newblock Rigorous results, Reprint of the 1989 edition.

\bibitem[Sin94]{Sin94}
Ya.~G. Sina\u{\i}.
\newblock {\em Topics in ergodic theory}, volume~44 of {\em Princeton
  Mathematical Series}.
\newblock Princeton University Press, Princeton, NJ, 1994.

\bibitem[Sin99]{Sin99}
Ya.~G. Sina\u{\i}.
\newblock Dynamics of a massive particle surrounded by a finite number of light
  particles.
\newblock {\em Teoret. Mat. Fiz.}, 121(1):110--116, 1999.

\bibitem[Sta65]{Statu}
V.~A. Statuljavi\v{c}us.
\newblock Limit theorems for densities and the asymptotic expansions for
  distributions of sums of independent random variables.
\newblock {\em Teor. Verojatnost. i Primene}, 10:645--659, 1965.

\bibitem[Sv65]{SiraSaha}
S.~H. Sira\v{z}dinov and N.~\v{S}aha\u{\i}darova.
\newblock On the uniform local theorem for densities.
\newblock {\em Izv. Akad. Nauk UzSSR Ser. Fiz.-Mat. Nauk}, 9(6):30--36, 1965.

\bibitem[\v{S}66]{Saha}
N.~\v{S}aha\u{\i}darova.
\newblock Uniform local and global theorems for densities.
\newblock {\em Izv. Akad. Nauk UzSSR Ser. Fiz.-Mat. Nauk}, 10(5):90--91, 1966.

\bibitem[\v{S}71]{Serva}
T.~L. \v{S}erva\v{s}idze.
\newblock The uniform estimation of the rate of convergence in a
  multidimensional local limit theorem for densities.
\newblock {\em Teor. Verojatnost. i Primenen.}, 16:765--767, 1971.

\end{thebibliography}

\end{document}